\documentclass[11pt]{article}
\usepackage{hyperref}
\usepackage{amssymb}
\usepackage{amsmath}
\usepackage{mdframed}
\usepackage{subfig}
\usepackage{cancel}
\usepackage{enumerate}
\usepackage{color}
\usepackage{mathtools}
\usepackage{pdfpages}
\usepackage{mathrsfs}
\usepackage{enumitem}
\usepackage{mathtools}
\usepackage{graphicx}
\usepackage{mdframed}
\usepackage{amsfonts}
\usepackage{cancel}
\usepackage{placeins}
\usepackage{amsthm}
\usepackage{tikz}
\usetikzlibrary{shapes,arrows,automata}
\usepackage{xcolor}
\usetikzlibrary{arrows.meta}

\definecolor{goldenpoppy}{rgb}{0.99, 0.76, 0.0}
\definecolor{richblack}{rgb}{0.06, 0.05, 0.03}
\definecolor{cadmiumred}{rgb}{0.89, 0.0, 0.13}
\definecolor{fuchsia}{rgb}{0.3, 0.0, 0.3}
\definecolor{green(ncs)}{rgb}{0.0, 0.52, 0.32}
\tikzstyle{species_T} = [circle,radius=0.1cm, text centered, draw=black, fill=goldenpoppy]
\tikzstyle{species_R} = [circle,radius=0.1cm, text centered, draw=black, fill=white]
\tikzstyle{species_C} = [circle,radius=0.1cm, text centered, draw=black, fill=fuchsia]
\tikzstyle{dots} = [circle,radius=0.1cm,text centered]
\tikzstyle{arrowA} = [-{Latex[length=2mm]},white,dashed]
\tikzstyle{arrowB} = [-{Latex[length=2mm]},green(ncs)]
\tikzstyle{arrowC} = [-{Latex[length=2mm]},cadmiumred]
\tikzstyle{inhibit} = [thick,-|,black!100]
\tikzstyle{loosely dashed}= [dash pattern=on 3pt off 6pt]

\newtheorem{definition}{Definition}[section]
\newtheorem{theorem}{Theorem}[section]

\newtheorem{example}{Example}[section]
\newtheorem{lemma}{Lemma}[section]
\newtheorem{corollary}{Corollary}[section]

\let\Item\item

\mdfdefinestyle{theoremstyle}{%
frametitlerule=true,%
innertopmargin=\topskip,
}
\mdtheorem[style=theoremstyle]{algorithm}{Algorithm}[section]

\begin{document}
\vspace*{-1cm}

\centerline{{\huge Chemical systems with chaos}}

\medskip
\bigskip

\centerline{
\renewcommand{\thefootnote}{$1$}
{\Large Tomislav Plesa \footnote{
Department of Applied Mathematics and Theoretical Physics, University of Cambridge,
Centre for Mathematical Sciences, Wilberforce Road, Cambridge, CB3 0WA, UK;
e-mail: tp525@cam.ac.uk}
}}

\medskip
\bigskip

\noindent
{\bf Abstract}: Three-dimensional 
polynomial dynamical systems (DSs) 
can display chaos with various properties  
already in the quadratic case with  
only one or two quadratic monomials. 
In particular, one-wing chaos is reported
in quadratic DSs with only one quadratic monomial,
while two-wing and hidden chaos in 
quadratic DSs with only two quadratic monomials. 
However, none of the reported DSs can be
realized with chemical reactions.
To bridge this gap, in this paper, we investigate 
chaos in chemical dynamical systems (CDSs)
- a subset of polynomial DSs that can model
the dynamics of mass-action chemical reaction networks.
To this end, we develop a fundamental theory 
for mapping polynomial DSs into CDSs of the same dimension
and with a reduced number of non-linear terms. 
Applying this theory, we show that, under suitable 
robustness assumptions, quadratic CDSs, 
and cubic CDSs with only one cubic, 
can display a rich set of chaotic solutions
already in three dimensions.
Furthermore, we construct some 
relatively simple three-dimensional examples,
including a quadratic CDS with one-wing chaos and three quadratics, 
a cubic CDS with two-wing chaos and one cubic,
and a quadratic CDS with hidden chaos and five quadratics.

\section{Introduction}
Systems of $N$ first-order 
autonomous ordinary-differential equations
with polynomials of at most degree $n$ on the right-hand sides, 
called $N$-dimensional $n$-degree polynomial \emph{dynamical systems} (DSs), 
can display a rich set of solutions
whose complexity increases with $N$ and $n$.
In particular, by the Poincar\'e-Bendixson 
theorem~\cite{Wiggins,Strogatz}, 
when $N \le 2$, every bounded solution
of this class of DSs in the long run converges
to an equilibrium (time-independent solution), 
a solution connecting equilibria, 
or a cycle (time-periodic solution).
On the other hand, if $N \ge 3$ and $n \ge 2$, then
the dynamics can become vastly more complicated: polynomial DSs
can then display solutions that are bounded and unstable, 
and that do not converge to equilibria, solutions connecting
equilibria or cycles~\cite{Wiggins,Strogatz}.
Under further suitable properties, 
a set containing such solutions is called a \emph{chaotic set};
if, furthermore, this set attracts all nearby solutions,
then it is called a \emph{chaotic attractor}~\cite{Wiggins,Strogatz,Chaos_Def}.
Containing unstable solutions, nearby solutions 
in a chaotic set may diverge from each 
other fast; consequently, an infeasible 
precision in the initial conditions may be needed
to accurately compute the solutions beyond 
some time, inspiring the name chaos.

A number of polynomial DSs with $N = 3$ and $n = 2$ 
are reported that, despite structural simplicity, display chaos
with various properties~\cite{Hidden}. Examples include 
three-dimensional quadratic DSs 
with only $1$ quadratic monomial, 
such as the R\"ossler system~\cite{Rossler}
and Sprott systems F--S~\cite{Sprott},
and with only $2$ quadratic monomials, 
such as the Lorenz system~\cite{Lorenz}
and Sprott systems A--E~\cite{Sprott}.
Topologically, all of these systems
have chaotic attractors with \emph{one wing},
except the Lorenz and Sprott systems B--C
which have chaotic attractors with \emph{two wings}.
The two-quadratic-terms DSs can also 
have a chaotic attractor coexisting with 
a unique and stable equilibrium~\cite{Sprott2};
such attractors are said to be \emph{hidden}, as they 
cannot be detected from the initial conditions close 
to the equilibrium. With some exceptions,
such as the Lorenz system, chaos is numerically demonstrated, 
but not proved to exist, in these examples.
In particular, proving existence of chaos
is in general difficult; for example, it took more than 
three decades to prove chaoticity of the Lorenz system~\cite{Tucker}.

In this paper, we focus on a special subset of 
polynomial DSs that can 
model the time-evolution of the 
(non-negative) concentrations of chemical species,
reacting under mass-action kinetics, called
\emph{chemical dynamical systems} (CDSs)~\cite{QCM};
in addition to chemistry, CDSs are also
used to model a range of 
biological phenomena~\cite{Janos,Feinberg}.
For example, $\mathrm{d} x/\mathrm{d} t = (1 - x)$
is a one-dimensional linear CDS which models
the time-evolution of the concentration $x = x(t)$ 
of chemical species $X$ reacting according to 
the production reaction $\varnothing \to X$ 
and the degradation reaction $X \to \varnothing$, 
where $\varnothing$ denotes some species not explicitly modelled.
In contrast, polynomial DS 
$\mathrm{d} x/\mathrm{d} t = -1$ is not chemical, 
because the monomial $-1$ drives $x(t)$ to negative values,
which are chemically infeasible. 

CDSs with chaos in the positive orthant 
are reported in the literature. 
However, compared to 
the general polynomial DSs, the reported chaotic 
CDSs are structurally more complicated - 
they have a higher dimension $N$, 
higher polynomial degree $n$, or a larger
number of non-linear terms. 
In particular, chaos is proven to exist in
a class of quadratic CDSs that are five-dimensional~\cite{Smale},
and in another class of quadratic CDSs 
with sufficiently high dimension~\cite{Vakulenko},
in three-dimensional quartic CDSs~\cite{Janos_chaos}
and, as proven more recently,  
in three-dimensional cubic CDSs
with exactly $2$ cubics~\cite{QCM}[Theorem 5.2]. 
Simpler systems are reported with 
more numerically-established chaos. 
In particular, the minimal Willamowski–R\"ossler system~\cite{RosslerW,MinRosslerW}
is a three-dimensional quadratic CDS 
with exactly $6$ quadratic monomials, 
put forward as having one-wing chaos;
a computer-assisted proof of some chaotic properties 
is presented in~\cite{RosslerW_proof}.
Another three-dimensional quadratic CDS with exactly $6$ quadratics
is constructed in~\cite{QCM}[Theorem E.1]; its chaoticity
is guaranteed assuming that Sprott system P~\cite{Sprott}[Table 1],
from which it is derived, has suitably robust chaos. 
To the best of the author's knowledge, 
no three-dimensional quadratic CDS with (one-wing) chaos
and $5$ or less quadratics, or cubic CDS with two-wing chaos
and less than $2$ cubics, are reported
at the time of writing this paper.
Furthermore, no three-dimensional quadratic or cubic CDSs with hidden 
chaos and a unique and stable equilibrium are reported; 
see also~\cite{CDS_Hidden}.

\begin{figure}[!htbp]
\vskip -0.0cm
\leftline{\hskip 
0.1cm (a) One-wing chaos~(\ref{eq:CDS_1}) \hskip  
1.8cm (b) Two-wing chaos~(\ref{eq:CDS_2}) \hskip  
2.4cm (c) Hidden chaos~(\ref{eq:CDS_3})}
\vskip 0.2cm
\centerline{
\hskip 0.0cm
\includegraphics[width=0.35\columnwidth]{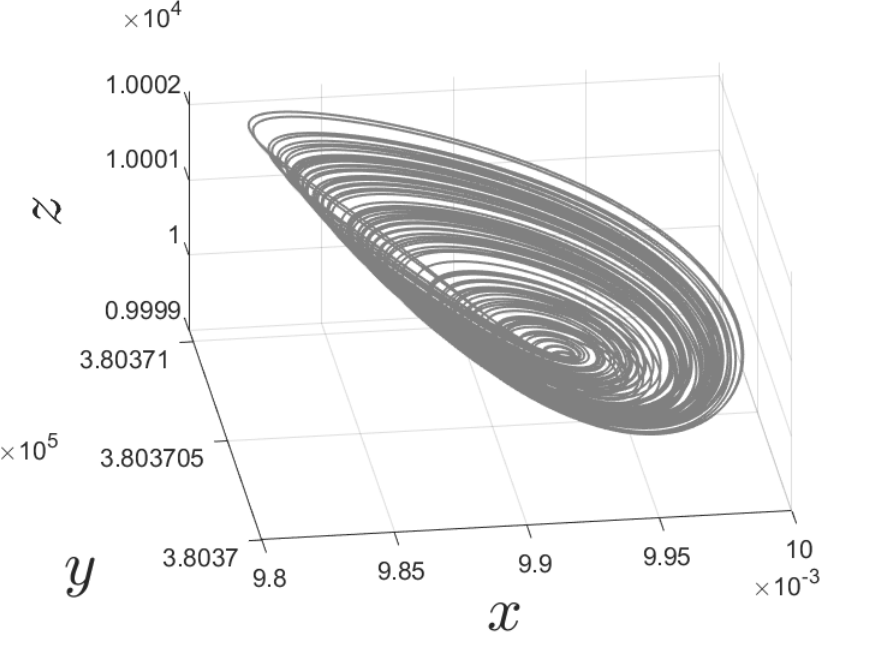}
\hskip 0.5cm
\includegraphics[width=0.35\columnwidth]{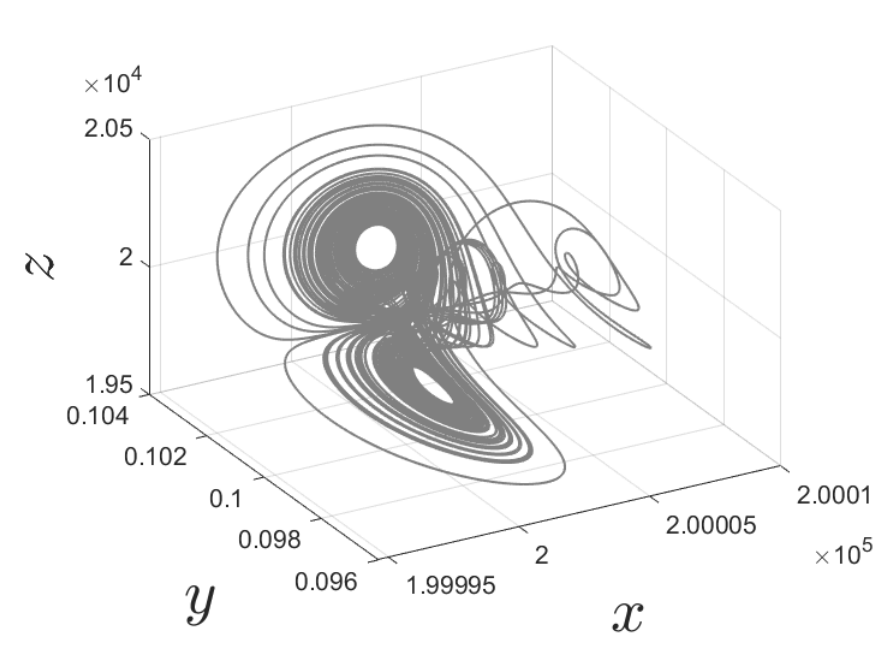}
\hskip 0.5cm
\includegraphics[width=0.35\columnwidth]{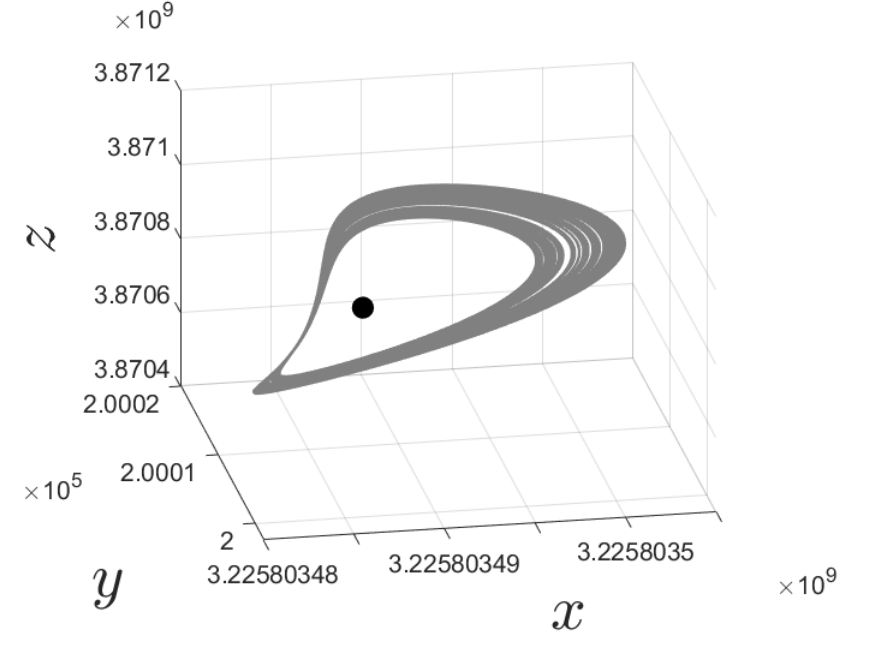}
}
\vskip -0.2cm
\caption{\it{\emph{Three-dimensional CDSs with chaos.} 
Panels \emph{(a)}--\emph{(c)} display solutions 
of the quadratic \emph{CDS}~$(\ref{eq:CDS_1})$
with exactly $3$ quadratic monomials, 
cubic \emph{CDS}~$(\ref{eq:CDS_2})$
with exactly $1$ cubic monomial,
and the quadratic \emph{CDS}~$(\ref{eq:CDS_3})$
with exactly $5$ quadratic monomials, respectively. 
The parameters and initial conditions 
are chosen as in \emph{Figures}~$\ref{fig:CDS_1}$--$\ref{fig:CDS_3}$, 
respectively.}} 
\label{fig:CDS_0}
\end{figure}

In this paper, we bridge this gap by developing
and applying some fundamental theory about 
mapping non-chemical into chemical DSs.
In particular, the so-called \emph{quasi-chemical maps} (QCMs)
are introduced in~\cite{QCM}, that can map 
polynomial DSs into CDSs purely by smoothly perturbing 
its right-hand side and translating the dependent variables,
thereby preserving both the dimension and robust structures. 
In this paper, we construct some novel QCMs that are refined
to introduce a lower number of non-linear terms, 
thus being suitable for designing structurally simpler CDSs. 
We then apply this theory to prove some results 
about the capacity of CDSs to have chaos. 
In particular, we prove that every quadratic DS with only 
$1$ quadratic monomial, and with suitably robust chaos, 
can be mapped to a \emph{quadratic} CDS of the same dimension 
and with chaos preserved.
Similarly, we show that every quadratic DS with only 
$2$ quadratic monomials can be mapped to a cubic CDS
with only $1$ cubic monomial, and that such CDSs 
can display chaos in three dimensions.
Furthermore, we construct some three-dimensional 
examples which appear to be the simplest reported
so far when it comes to the number of the highest-degree monomials:
a quadratic CDS with a one-wing
chaotic attractor and exactly $3$ quadratics, 
a cubic CDS with a two-wings chaotic attractor and exactly $1$ cubic,
and a quadratic CDS with a hidden chaotic attractor, and unique and stable equilibrium, with exactly $5$ quadratics; see Figure~\ref{fig:CDS_0}.

The paper is organized as follows. 
In Section~\ref{sec:background}, we present 
some background theory; further details are presented in Appendices~\ref{app:CRN}
and~\ref{app:LCE}. In Section~\ref{sec:DS_to_CDS}, 
we present some theory about mapping DSs into 
structurally relatively simple CDSs using QCMs. 
This theory is then applied to study chaos
in CDSs in Section~\ref{sec:applications}. Finally, 
we provide a summary and discussion in 
Section~\ref{sec:discussion}.

\section{Background}
\label{sec:background}
In this section, we present some notation and background theory.

\textbf{Notation}.
The spaces of real, non-negative and 
positive real numbers are denoted by $\mathbb{R}$,
$\mathbb{R}_{\ge}$ and $\mathbb{R}_{>}$, respectively.
Absolute value of $x \in \mathbb{R}$ is denoted by $|x|$.
Euclidean column vectors are denoted by 
$\mathbf{x} = (x_1, x_2, \ldots, x_N)^{\top} \in \mathbb{R}^{N}$, 
where $\cdot^{\top}$ is the transpose operator.
The Euclidean norm of $\mathbf{x} \in \mathbb{R}^{N}$ is given by 
$\|\mathbf{x}\| = (x_1^2 + x_2^2 + \ldots + x_N^2)^{1/2}$
and, abusing the notation slightly, 
we also denote by $\|\cdot\|$ the induced matrix norm.
Operator nabla is given by $\nabla
= (\partial/\partial x_1, \partial/\partial x_2,
\ldots, \partial/\partial x_N)$.

\subsection{Dynamical systems}
Consider a system of autonomous
ordinary-differential equations
\begin{align}
\frac{\mathrm{d} \mathbf{x}}{\mathrm{d} t}
& = \mathbf{f}(\mathbf{x}; n),
\label{eq:DS}
\end{align}
where $t \in \mathbb{R}$ is time,
$\mathbf{x} = (x_1,x_2,\ldots,x_N)^{\top} \in \mathbb{R}^N$
and $\mathbf{f}(\mathbf{x}; n) = (f_1(\mathbf{x};n),
f_2(\mathbf{x};n), \ldots, f_N(\mathbf{x};n))^{\top} \in \mathbb{R}^N$
is a vector field with $f_i(\cdot; n) : \mathbb{R}^N \to \mathbb{R}$
a polynomial of degree at most $n$ for all $i = 1,2,\ldots,N$.
We say that~(\ref{eq:DS}) is an $N$-dimensional 
$n$-degree polynomial \emph{dynamical system} (DS).
For simplicity, we assume that all the variables are dimensionless;
furthermore, we assume that every equation in~(\ref{eq:DS}) 
is written so that all of its monomials are distinct.

Key to this paper is a special subset of polynomial DSs.

\begin{definition} $($\textbf{Chemical dynamical system}$)$ 
\label{def:CDS}
Assume that $m(\mathbf{x})$ is a monomial in $f_i(\mathbf{x})$
such that $m(\mathbf{x}) \ge 0$ 
when $x_i = 0$ and $x_{j} \ge 0$ for all $j \ne i$.
Then $m(\mathbf{x})$ is said to be a \emph{chemical} monomial.
If $f_i(\mathbf{x})$ contains only chemical monomials,
then $f_i(\mathbf{x})$ is \emph{chemical}; 
we then also say that the $i$th equation from~$(\ref{eq:DS})$
is \emph{chemical}.
If every equation of~$(\ref{eq:DS})$
is chemical, then~$(\ref{eq:DS})$ is said to be an 
$N$-dimensional $n$-degree (mass-action) 
\emph{chemical dynamical system (CDS)}.
\end{definition}
 
If $f_i(\mathbf{x})$ is chemical, then 
$f_i(\mathbf{x}) \ge 0$ when $x_i = 0$ 
and $x_{j} \ge 0$ for all $j \ne i$; consequently, 
the CDS solution $\mathbf{x}(t;\mathbf{x}_0) \in \mathbb{R}_{\ge}^N$
is non-negative for all $t \ge 0$, provided 
that is it non-negative initially, 
$\mathbf{x}(0;\mathbf{x}_0) 
= \mathbf{x}_0 \in \mathbb{R}_{\ge}^N$.
Furthermore, variables $x_1,x_2,\ldots,x_N$ 
from CDSs can be interpreted 
as (non-negative) concentrations of chemical species, 
which react according to a set
of reactions jointly called a \emph{chemical reaction network} (CRN)~\cite{Janos,Feinberg}; 
see Appendix~\ref{app:CRN} for more details.

To characterize structural complexity of CDSs and CRNs, 
we introduce the following definition.

\begin{definition} $($\textbf{Structural complexity}$)$ 
\label{def:CDS_complexity}
If $\mathbf{f}(\mathbf{x}; n)$ contains exactly $M_i$
monomials of degree $i$, 
then \emph{DS}~$(\ref{eq:DS})$ is said to be 
an $n$-degree $(M_0 + M_1 + \ldots + M_n,M_2,M_3,\ldots,M_n)$ \emph{DS}. 
If a \emph{CRN} induced by an $n$-degree \emph{CDS} 
contains exactly $R_i$ reactions of degree $i$,
then it is said to be an $n$-degree
$(R_0 + R_1 + \ldots + R_n,R_2,R_3,\ldots,R_n)$ \emph{CRN}.
\end{definition}

\noindent
\textbf{Remark}. CDSs induce CRNs with $R_i \le M_i$; 
see Appendix~\ref{app:CRN} and the examples below.

\emph{Transformation of variables}. 
It is of great interest to analyze how CDSs
behave under various transformations of their variables.
For example, CDSs are not invariant under general 
affine transformations~\cite{QCM}.
In particular, translating the dependent variables
via $\mathbf{x} \to (\mathbf{x} - \mathbf{T})$,
defined as introducing the new variables
$\mathbf{\bar{x}} = (\mathbf{x} + \mathbf{T})$ in~(\ref{eq:DS}),
in general transforms chemical DSs into non-chemical ones, 
see also Section~\ref{sec:DS_to_CDS}. 
Let us now consider two special
affine maps: rescaling $x_i \to s_i x_i$,
defined as introducing a new variable 
$\bar{x}_i = x_i/s_i$ in~(\ref{eq:DS}), 
and permutation $x_i \to x_j$, defined via new variables 
$\bar{x}_i = x_j$ and $\bar{x}_j = x_i$.
It follows directly from Definition~\ref{def:CDS}  
that CDSs are invariant under every rescaling with positive factors,
and under every permutation, which we now formalize. 

\begin{lemma}$($\textbf{\emph{Positive scaling and permutation}}$)$ 
\label{lemma:scaling_permutation}
Assume that \emph{DS}~$(\ref{eq:DS})$ is chemical 
(respectively, non-chemical). Then, under 
every rescaling $x_i \to s_i x_i$ with $s_i > 0$,
and under every permutation $x_i \to x_j$, 
\emph{DS}~$(\ref{eq:DS})$ remains chemical 
(respectively, non-chemical).
\end{lemma}

In what follows, when convenient, we permute 
and rescale the variables with suitable positive factors, 
which does not change the chemical nature of DSs
by Lemma~$\ref{lemma:scaling_permutation}$.
Furthermore, in $N = 3$ dimensions, 
we let $x = x_1$, $y = x_2$ and $z = x_3$.

\begin{example} 
Consider the three-dimensional quadratic \emph{DS}
\begin{align}
\frac{\mathrm{d} x}{\mathrm{d} t}
& = \frac{1}{5} - \frac{57}{10} x + x y, \nonumber \\
\frac{\mathrm{d} y}{\mathrm{d} t}
& = -x - z, \nonumber \\
\frac{\mathrm{d} z}{\mathrm{d} t}
& = y + \frac{1}{5} z,
\label{eq:Rossler}
\end{align}
obtained via permutation $x \to z$ and then $y \to z$
in R\"ossler system~\emph{\cite{Rossler}}. 
By \emph{Definition}~$\ref{def:CDS_complexity}$,
since $M_0 = 1$, $M_1 = 5$ and $M_2 = 1$,
$(\ref{eq:Rossler})$ is a $(7,1)$ \emph{DS}.
Function $(-x - z)$ contains two non-chemical monomials;
hence, by \emph{Definition}~$\ref{def:CDS}$, 
$(\ref{eq:Rossler})$ is not a \emph{CDS},
and there is no associated \emph{CRN}.
\end{example}

\begin{example}
Consider the three-dimensional quadratic $(9,6)$ \emph{DS}
\begin{align}
\frac{\mathrm{d} x}{\mathrm{d} t}
& = 30 x - \frac{1}{2} x^2 - x y - x z, \nonumber \\
\frac{\mathrm{d} y}{\mathrm{d} t}
& = \frac{33}{2} y - \frac{1}{2} y^2 - x y, \nonumber \\
\frac{\mathrm{d} z}{\mathrm{d} t}
& = - 10 z +  x z, 
\label{eq:Willamowski_Rossler_CDS}
\end{align}
obtained from the minimal Willamowski–R\"ossler 
system~\emph{\cite{RosslerW,MinRosslerW}}
via permutation $y \to z$.
Since each of its equations is chemical,
$(\ref{eq:Willamowski_Rossler_CDS})$ is a \emph{CDS}.
Denoting by $X, Y, Z$ the chemical species with concentrations
$x, y, z$, an associated \emph{CRN} reads
\begin{align}
X & \xrightarrow[]{30} 2 X, \; \; \; 
2 X \xrightarrow[]{1/2} X,  \; \; \; 
X + Y \xrightarrow[]{1} Y, \; \; \; 
X + Z \xrightarrow[]{1} Z, \nonumber \\
Y & \xrightarrow[]{33/2} 2 Y, \; \; \; 
2 Y \xrightarrow[]{1/2} Y, \; \; \; 
X + Y \xrightarrow[]{1} X, \; \; \; 
Z \xrightarrow[]{10} \varnothing, \; \; \; 
X + Z \xrightarrow[]{1} X + 2 Z,
\label{eq:Willamowski_Rossler_canon_CRN}
\end{align}
where $\varnothing$ denotes some species
not explicitly modelled. By \emph{Definition}~$\ref{def:CDS_complexity}$,
since $R_0 = 0$, $R_1 = 3$ and $R_2 = 6$,
$(\ref{eq:Willamowski_Rossler_canon_CRN})$ is a
$(9,6)$ \emph{CRN}. This network is called
the \emph{canonical CRN} of~$(\ref{eq:Willamowski_Rossler_CDS})$;
there also exist other, non-canonical, \emph{CRNs}.
In particular, since the terms $x y$ 
in the first and second equations
of~$(\ref{eq:Willamowski_Rossler_CDS})$ are multiplied,
up to sign, by the same coefficient, 
the two canonical reactions $X + Y \xrightarrow[]{1} Y$
and $X + Y \xrightarrow[]{1} X$ can be \emph{fused}
into a single reaction $X + Y \xrightarrow[]{1} \varnothing$.
Similarly, $X + Z \xrightarrow[]{1} Z$ 
and $X + Z \xrightarrow[]{1} X + 2 Z$ can be fused into
$X + Z \xrightarrow[]{1} 2 Z$. Using these two fused
reactions, one obtains a non-canonical $(7,4)$ \emph{CRN}
\begin{align}
X & \xrightarrow[]{30} 2 X, \; \; \; 
2 X \xrightarrow[]{1/2} X,  \; \; \; 
X + Y \xrightarrow[]{1} \varnothing, \nonumber \\
X + Z & \xrightarrow[]{1} 2 Z, \; \; \; 
Y \xrightarrow[]{33/2} 2 Y, \; \; \; 
2 Y \xrightarrow[]{1/2} Y, \; \; \; 
Z \xrightarrow[]{10} \varnothing.
\label{eq:Rossler_CRN}
\end{align}
See \emph{Appendix~\ref{app:CRN}} for 
more details on how to construct \emph{CRN}s
for any given \emph{CDS}.
\end{example}

\subsection{Chaos} 
Let us now discuss what is meant by chaos, 
starting with the following definition.
\begin{definition} $($\textbf{Invariant set}$)$ 
\label{def:invariant_set} 
Consider a set $\mathbb{V} \subset \mathbb{R}^N$ 
in the state-space $\mathbb{R}^N$ of \emph{DS}~$(\ref{eq:DS})$.
Assume that the following statement holds:
if $\mathbf{x}(0;\mathbf{x}_0) = \mathbf{x}_0 \in \mathbb{V}$, then 
the solution $\mathbf{x}(t;\mathbf{x}_0) \in \mathbb{V}$ 
for all $t \in \mathbb{R}$.
Then, $\mathbb{V}$ is said to be an
\emph{invariant set} of \emph{DS}~$(\ref{eq:DS})$.
Furthermore, if there exists $M > 0$ such that 
$\textrm{\emph{sup}}_{t \in \mathbb{R}} 
\|\mathbf{x}(t; \mathbf{x}_0)\| \le M$ 
for every $\mathbf{x}_0 \in \mathbb{V}$, 
then the invariant set $\mathbb{V}$ is said to be bounded.
\end{definition} 

Broadly speaking, a bounded invariant set $\mathbb{V}$ 
for DS~(\ref{eq:DS}) is \emph{chaotic} if it has the 
following two properties~\cite{Wiggins}:
$(\mathcal{P}_1)$ trajectories in $\mathbb{V}$ 
are ``sensitive to initial conditions'', and
$(\mathcal{P}_2)$ $\mathbb{V}$ is ``irreducible'';
$\mathbb{V}$ is said to be 
a \emph{chaotic attractor} if additionally 
($\mathcal{P}_3$) there exists a neighborhood 
$\mathbb{U} \supset \mathbb{V}$ 
such that if $\mathbf{x}(0;\mathbf{x}_0) = \mathbf{x}_0 \in \mathbb{U}$
then $\mathbf{x}(t;\mathbf{x}_0)$ approaches 
$\mathbb{V}$ as $t \to \infty$. 
Properties $(\mathcal{P}_1)$ and $(\mathcal{P}_2)$ 
can be defined in a number of non-equivalent 
ways~\cite{Chaos_Def}. For example,
$(\mathcal{P}_1)$ can be defined
via existence of a positive 
\emph{Lyapunov characteristic exponent} (LCE)~\cite{Lyapunov,Adrianova},
which implies that nearby trajectories in $\mathbb{V}$ separate
exponentially fast; see Appendix~\ref{app:LCE} for more details.
Property $(\mathcal{P}_2)$ can be defined 
via existence of a trajectory dense in $\mathbb{V}$.
Note that some definitions of chaos
impose additional properties, 
such as existence of periodic trajectories 
that are dense in $\mathbb{V}$,
while other definitions relax certain properties, 
such as not requiring that set 
$\mathbb{V}$ is bounded~\cite{Chaos_Def}.

Definitions of chaos implicitly or explicitly 
exclude equilibria, trajectories connecting equilibria 
and (quasi-)periodic trajectories
from being considered chaotic sets.
Since these are the only candidates for chaos 
in linear DSs of any finite dimension, 
and in two-dimensional DSs of any degree
by the Poincar\'e-Bendixson theorem~\cite{Wiggins}, 
it follows that~(\ref{eq:DS}) can display chaos 
only if dimension $N \ge 3$ and degree $n \ge 2$.
In this context, a number of three-dimensional quadratic DSs
with only one quadratic term have been presented
as ``minimal'' chaotic systems, such as 
the $(7,1)$ R\"ossler system~(\ref{eq:Rossler})
and fourteen $(6,1)$ Sprott systems F--S~\cite{Sprott}[Table 1].

\textbf{Robustness}. 
In this paper, we allow any definition of chaos
as a bounded invariant set
with some properties $\mathcal{P}$;
the only restriction that we impose is that this set, 
and its properties, are suitably robust. To formulate this, 
let us consider a perturbation of DS~(\ref{eq:DS}), given by
\begin{align}
\frac{\mathrm{d} \mathbf{x}}{\mathrm{d} t}
& = \mathbf{f}(\mathbf{x}; n) + \mathbf{p}(\mathbf{x}).
\label{eq:DS_perturbed}
\end{align}

\begin{definition} $($\textbf{Robustness}$)$ 
\label{def:robustness}
Let $\mathbb{V} \subset \mathbb{R}^N$ 
be a bounded invariant set with properties $\mathcal{P}$ 
for \emph{DS}~$(\ref{eq:DS})$.
Assume that there exists a compact set 
$\mathbb{U} \supset \mathbb{V}$ 
such that for every sufficiently small $\varepsilon > 0$ and
every continuously differentiable function
$\mathbf{p}(\mathbf{x})$
with $\textrm{\emph{max}}_{\mathbf{x} \in \mathbb{U}} 
(\|\mathbf{p}(\mathbf{x})\| + 
\|\nabla \mathbf{p}(\mathbf{x})\|) < \varepsilon$
the perturbed \emph{DS}~$(\ref{eq:DS_perturbed})$
has a bounded invariant set $\mathbb{V}' \subset \mathbb{U}$ 
with the same properties $\mathcal{P}$.
Then, invariant set $\mathbb{V}$ of~$(\ref{eq:DS})$ 
and its properties $\mathcal{P}$ are said to be \emph{robust}.
\end{definition}

\section{Theory: Quasi-chemical maps}
\label{sec:DS_to_CDS}
Assume we are given a non-chemical 
DS~(\ref{eq:DS}) with a desired robust invariant set 
in $\mathbb{R}^N$ (see Definition~\ref{def:robustness}), 
and we wish to construct a CDS 
(see Definition~\ref{def:CDS}) of the same dimension and with 
such an invariant set in $\mathbb{R}_{>}^N$. 
To position this set into the positive orthant,
one can appropriately translate 
the dependent variables in~(\ref{eq:DS}).
However, translations (and, more broadly,
affine transformations) alone
do not in general ensure that 
the resulting DS is chemical~\cite{QCM} 
- another non-linear map must in general be applied 
that also preserves robust structures.
A candidate map involves perturbing DS~(\ref{eq:DS})
before suitable translations are applied.

\begin{definition}$($\textbf{\emph{Quasi-chemical map}}$)$ 
\label{def:QCM}
Assume that for polynomial functions
$f_1(\mathbf{x}),f_2(\mathbf{x}),\ldots,$ $f_N(\mathbf{x})$ 
there exist polynomials $p_1(\mathbf{x}; \boldsymbol{\mu}),
p_2(\mathbf{x}; \boldsymbol{\mu}),
\ldots, p_N(\mathbf{x}; \boldsymbol{\mu})$, 
with coefficients smooth in parameters 
$\boldsymbol{\mu} \in \mathbb{R}_{>}^{p}$, 
and a vector $\mathbf{T}(\boldsymbol{\mu}) \in \mathbb{R}_{>}^N$,
satisfying the following properties
for every $i = 1,2,\ldots, N$.
\begin{enumerate}
\item[$(i)$](\textbf{Small perturbations}) 
For every compact set 
$\mathbb{U} \subset \mathbb{R}^N$, 
and for every sufficiently small $\|\boldsymbol{\mu}\|$,
$\textrm{\emph{max}}_{\mathbf{x} \in \mathbb{U}} 
|p_i(\mathbf{x}; \boldsymbol{\mu})|$ 
is arbitrarily small.
\item[$(ii)$] (\textbf{Large translations}) 
For every sufficiently small $\|\boldsymbol{\mu}\|$,
$\|\mathbf{T}(\boldsymbol{\mu})\|$ is arbitrarily large. 
\item[$(iii)$] (\textbf{Chemicality}) 
For every sufficiently small $\|\boldsymbol{\mu}\|$,
the perturbed function 
$f_i(\mathbf{x}) + p_i(\mathbf{x}; \boldsymbol{\mu})$
becomes chemical under the translation
$\mathbf{x} \to (\mathbf{x} - \mathbf{T}(\boldsymbol{\mu}))$.
\end{enumerate}
Then, the map from \emph{DS}~$(\ref{eq:DS})$
to the \emph{CDS} $\mathrm{d} \bar{\mathbf{x}}/\mathrm{d} t = 
\mathbf{f}(\bar{\mathbf{x}} - \mathbf{T}(\boldsymbol{\mu})) + 
\mathbf{p}(\bar{\mathbf{x}} - \mathbf{T}(\boldsymbol{\mu}); 
\boldsymbol{\mu})$ 
is for every sufficiently small $\|\boldsymbol{\mu}\|$
called a \emph{quasi-chemical map 
(QCM)} induced by the perturbations 
$p_1(\mathbf{x}; \boldsymbol{\mu}),$ 
$p_2(\mathbf{x}; \boldsymbol{\mu}),
\ldots, p_N(\mathbf{x}; \boldsymbol{\mu})$, 
and the translation vector 
$\mathbf{T}(\boldsymbol{\mu})$.
\end{definition}

\noindent 
\textbf{Remark}. Consisting only of 
(i) arbitrarily small smooth perturbations, and 
(ii) translations, QCMs preserve robust structures.
Furthermore, since translations can be arbitrarily large, 
every bounded set $\mathbb{U} \subset \mathbb{R}^N$
of DS~(\ref{eq:DS}) can be mapped into $\mathbb{R}_{>}^N$ via QCMs.

This family of non-linear maps is introduced in~\cite{QCM}, 
where it is shown that a particular QCM
\emph{universally} maps polynomial DSs into CDSs. 
In other words, while it is in general not possible to
map a given polynomial DS to a CDS via only translations, 
it is always possible to do so for 
a nearby perturbed polynomial DS; 
we now present this result rigorously.

\begin{theorem}$($\textbf{\emph{Universal QCM}}$)$ 
\label{theorem:universal}
Every $N$-dimensional $n$-degree \emph{DS}~$(\ref{eq:DS})$,
with exactly $M_n$ monomials of degree $n$,
can be mapped via a \emph{QCM} to 
an $N$-dimensional $(n+1)$-degree \emph{CDS},
with exactly $M_{n + 1} = M_n$ monomials of degree $(n+1)$. 
\end{theorem}

\begin{proof}
This result is given as~\cite{QCM}[Theorem 4.2];
for reader's convenience, we reproduce the proof. 
In particular, perturbing~(\ref{eq:DS}) via
\begin{align}
\frac{\mathrm{d} x_i}{\mathrm{d} t}
& = f_i(\mathbf{x}) + \frac{\mu}{a_i} x_i f_i(\mathbf{x}), 
\; \; \; i = 1,2,\ldots, N, \label{eq:CDSs_general_1}
\end{align}
for every $\mathbf{a} = (a_1,a_2,\ldots,a_N)^{\top} \in \mathbb{R}_{>}^N$ 
and every $\mu > 0$ sufficiently small,
and then translating via
$\bar{\mathbf{x}} = \mathbf{x} + \mathbf{a}/\mu$, 
one obtains the $N$-dimensional $(n+1)$-degree CDS
\begin{align}
\frac{\mathrm{d} \bar{x}_i}{\mathrm{d} t}
& = \frac{\mu}{a_i} \bar{x}_i f_i 
\left(\bar{\mathbf{x}} - \frac{\mathbf{a}}{\mu}\right), 
\; \; \; i = 1,2,\ldots, N, \label{eq:CDSs_general}
\end{align}
with $M_n$ monomials of degree $(n+1)$.
By Definition~\ref{def:QCM}, 
for every $\mathbf{a} \in \mathbb{R}_{>}^N$
there exists $\mu_0 > 0$ such that
for every $\mu \in (0,\mu_0)$ the map from 
DS~(\ref{eq:DS}) to the CDS~(\ref{eq:CDSs_general})
is a QCM induced by the perturbations
$p_i(\mathbf{x}; \mu) = \mu x_i f_i(\mathbf{x})/a_i$,
$i = 1,2,\ldots,N$, and the translation vector 
$\mathbf{T}(\mu) = \mathbf{a}/\mu$.
\end{proof}

\textbf{Splitting}. To systematically construct QCMs, 
let us split the $i$th equation from~(\ref{eq:DS}) as follows:
\begin{align}
\frac{\mathrm{d} x_i}{\mathrm{d} t}
& = f_i(\mathbf{x}) =
\sum_{j = 1}^M f_{i,j}(\mathbf{x}),
\label{eq:splitting}
\end{align}
where $f_{i,j}(\mathbf{x})$ are arbitrary polynomials.
Assume that one constructs a perturbation
and translation satisfying Definition~\ref{def:QCM}(i)--(iii) 
for each of the functions $f_{i,j}(\mathbf{x})$ separately. 
If all of the translation vectors are identical, 
then, by simply adding all of the perturbations
to the full equation~(\ref{eq:splitting}),
one obtains a chemical equation under the common translation.
We now formalize this basic result.

\begin{lemma} $($\textbf{\emph{Splitting lemma}}$)$ 
\label{lemma:splitting}
Assume that, for every $j = 1,2,\ldots,M$,
the function 
$f_{i,j}(\mathbf{x}) + p_{i,j}(\mathbf{x}; \boldsymbol{\mu})$ 
becomes chemical under the translation 
$\mathbf{x} \to (\mathbf{x} - \mathbf{T}(\boldsymbol{\mu}))$ 
for every sufficiently small $\|\boldsymbol{\mu}\|$. 
Let $f_i(\mathbf{x})$ be the function defined 
in~$(\ref{eq:splitting})$,
and $p_i(\mathbf{x}; \boldsymbol{\mu}) = 
\sum_{j = 1}^M p_{i,j}(\mathbf{x}; \boldsymbol{\mu})$.
Then, the function 
$f_i(\mathbf{x}) + p_i(\mathbf{x}; \boldsymbol{\mu})$ 
also becomes chemical under 
$\mathbf{x} \to (\mathbf{x} - \mathbf{T}(\boldsymbol{\mu}))$
for every sufficiently small $\|\boldsymbol{\mu}\|$.
\end{lemma}

\begin{proof}
Applying the translation 
$\bar{\mathbf{x}} = \mathbf{x} + \mathbf{T}(\boldsymbol{\mu})$
on $f_i(\mathbf{x}) + p_i(\mathbf{x}; \boldsymbol{\mu})$, 
one obtains 
\begin{align}
f_i \left(\bar{\mathbf{x}} - \mathbf{T}(\boldsymbol{\mu})\right) 
+ p_i\left(\bar{\mathbf{x}} - \mathbf{T}(\boldsymbol{\mu}); 
\boldsymbol{\mu}\right)
& = \sum_{j = 1}^M \left[f_{i,j}\left(\bar{\mathbf{x}} - \mathbf{T}(\boldsymbol{\mu})\right) + p_{i,j}\left(\bar{\mathbf{x}} - \mathbf{T}(\boldsymbol{\mu}); \boldsymbol{\mu}\right) \right].
\label{eq:splitting_lemma}
\end{align}
By assumption, each of the summands is chemical
for every sufficiently small $\|\boldsymbol{\mu}\|$.
Being a sum of chemical functions,
the left-hand side in~(\ref{eq:splitting_lemma})
is then also chemical by Definition~\ref{def:CDS}.
\end{proof}

The QCM from Theorem~\ref{theorem:universal}
adds perturbations of identical form, 
with a single perturbation parameter, 
to every equation, and is applicable to every polynomial DS;
however, this universality comes at a cost: 
a larger number of non-linear monomials can be introduced. 
In this section, we construct QCMs tailored 
to introduce a lower number 
of quadratic and cubic monomials;
to achieve this, perturbations 
of different form in general, with multiple perturbation parameters,
are added to different equations.
In particular, in Section~\ref{sec:linear}, 
we first construct some QCMs refined to introduce 
a lower number of quadratics when applied on linear DSs.
Then, in Section~\ref{sec:quadratic}, we present 
some QCMs that introduce a lower number of 
cubics when applied on quadratic DSs.
To reduce the number of both quadratics and cubics,
these two classes of QCMs can be combined
via suitable splittings as per Lemma~\ref{lemma:splitting}.

\subsection{Linear DSs}
\label{sec:linear}
Consider DS~(\ref{eq:DS}) with $n = 1$, which can be written as
\begin{align}
\frac{\mathrm{d} x_i}{\mathrm{d} t}
& = f_i(\mathbf{x}; 1) 
= \alpha_{i,0} + \sum_{j = 1}^N \alpha_{i,j} x_j,
\; \; \; i = 1,2,\ldots,N.
\label{eq:DS_linear}
\end{align}
Let us consider a special case,
when the $i$th equation takes the form
\begin{align}
\frac{\mathrm{d} x_i}{\mathrm{d} t}
& = l_i(\mathbf{x}) 
= \alpha_{i,0} 
+ \alpha_{i,i} x_i
+ \sum_{j \in I_i} \alpha_{i,j} x_j,
\label{eq:i_linear}
\end{align}
where the set $I_i$ contains only indices $j \ne i$ 
such that $\alpha_{i,j} > 0$.

\begin{lemma}$($\textbf{\emph{Chemicality for~$(\ref{eq:i_linear})$}}$)$ 
\label{lemma:DS_linear_1}
For every $\varepsilon > 0$
and $\mathbf{a} \in \mathbb{R}_{>}^N$
there exists $\mu_0 > 0$ such that for 
every $\mu \in (0,\mu_0)$ the perturbed equation
\begin{align}
\frac{\mathrm{d} x_i}{\mathrm{d} t}
& = \alpha_{i,0} 
+ \alpha_{i,i} x_i 
+ \sum_{j \in I_i} \alpha_{i,j} x_j
+ \varepsilon x_i^2,
\label{eq:i_linear_p1}
\end{align}
becomes chemical under the translation
$\mathbf{x} \to (\mathbf{x} - \mathbf{a}/\mu)$.
\end{lemma}

\begin{proof}
Translating via $\bar{\mathbf{x}} = \mathbf{x} + \mathbf{a}/\mu$
in~(\ref{eq:i_linear_p1}), one obtains the equation
\begin{align}
\frac{\mathrm{d} \bar{x}_i}{\mathrm{d} t}
& = \left[\frac{\varepsilon a_i^2}{\mu^2} - \frac{1}{\mu} 
\left(\alpha_{i,i} a_i + \sum_{j \in I_i} \alpha_{i,j} a_j \right) 
+  \alpha_{i,0} \right] 
+ \left(\alpha_{i,i} - \frac{2 \varepsilon a_i}{\mu} \right) \bar{x}_i 
+ \sum_{j \in I_i} \alpha_{i,j} \bar{x}_j
 + \varepsilon \bar{x}_i^2,
\label{eq:case_2_2}
\end{align} 
which is chemical for every sufficiently small $\mu > 0$ 
by Definition~\ref{def:CDS}. 
\end{proof}

By Theorem~\ref{theorem:universal}, 
DS~(\ref{eq:DS_linear}) is mapped via the universal QCM to 
a quadratic CDS with exactly $M_1$
quadratics. Using a suitable splitting 
as per Lemma~\ref{lemma:splitting},
we now construct a QCM that 
can lead to a lower number of quadratics. 
To this end, let $M_{1,i}^{-}$
be the total number of terms 
$\alpha_{i,j} x_j$ with $j \ne i$
and $\alpha_{i,j} < 0$ in the $i$th 
equation of~(\ref{eq:DS_linear}); 
we then let $M_1^{-} = \sum_{i = 1}^N M_{1,i}^{-}$.

\begin{theorem}
\label{theorem:DS_linear}
Every $N$-dimensional linear \emph{DS}~$(\ref{eq:DS_linear})$,
with exactly $M_1^{-}$ non-chemical first-degree monomials,
can be mapped via a \emph{QCM} to 
an $N$-dimensional quadratic \emph{CDS}
with $M_2 \le (M_1^{-} + N)$ quadratic monomials.
\end{theorem}

\begin{proof}
Let us split DS~(\ref{eq:DS_linear}) as follows:
\begin{align}
\frac{\mathrm{d} x_i}{\mathrm{d} t}
& = f_i(\mathbf{x}; 1) = l_i(\mathbf{x}) + r_i(\mathbf{x}),
\; \; \; i = 1,2,\ldots, N,
\label{eq:DS_linear_split}
\end{align}
with $l_i(\mathbf{x})$ as in~(\ref{eq:i_linear}),
with set $I_i$ containing \emph{all} the 
indices $j \ne i$ from~(\ref{eq:DS_linear}) 
such that $\alpha_{i,j} > 0$,
and the remainder 
$r_i(\mathbf{x}) = f_i(\mathbf{x}; 1) - l_i(\mathbf{x})$.
Then, for every sufficiently small $\varepsilon > 0$
and for every $\mathbf{a} \in \mathbb{R}_{>}^N$
there exists $\mu_0 > 0$ such that for 
every $\mu \in (0,\mu_0)$ the perturbed equations
\begin{align}
\frac{\mathrm{d} x_i}{\mathrm{d} t}
& = l_i(\mathbf{x}) + r_i(\mathbf{x})
+ \varepsilon x_i^2 + \frac{\mu}{a_i} x_i r_i(\mathbf{x}),
\; \; \; i = 1,2,\ldots, N,
\label{eq:DS_linear_split}
\end{align}
become chemical under the translation 
$\mathbf{x} \to (\mathbf{x} - \mathbf{a}/\mu)$ 
by Lemma~\ref{lemma:DS_linear_1}, 
Theorem~\ref{theorem:universal}
and Lemma~\ref{lemma:splitting}. 
By Definition~\ref{def:QCM}, 
the perturbations 
$p_i(\mathbf{x}; \varepsilon, \mu)
= \varepsilon x_i^2 + \mu x_i r_i(\mathbf{x})/a_i$, 
$i = 1,2,\ldots,N$,
and the translation vector $\mathbf{a}/\mu$,
then induce a QCM for DS~(\ref{eq:DS_linear}). 
Every equation from~$(\ref{eq:DS_linear_split})$
has $(M_{1,i}^{-} + 1)$ quadratics;
hence, the corresponding CDS has 
$M_2 = \sum_{i = 1}^N (M_{1,i}^{-} + 1) 
= (M_{1}^{-} + N)$ quadratics.
\end{proof}

The perturbation from Lemma~\ref{lemma:DS_linear_1}
is quadratic. We now show that, if 
$\alpha_{i,i} \le 0$ in~(\ref{eq:i_linear}), 
then a linear perturbation is sufficient.

\begin{lemma}$($\textbf{\emph{Chemicality for~$(\ref{eq:i_linear})$
with $\alpha_{i,i} \le 0$}}$)$ 
\label{lemma:DS_linear_2}
Assume that $\alpha_{i,i} \le 0$ in~$(\ref{eq:i_linear})$.
Assume also that parameters $\mathbf{a} \in \mathbb{R}_{>}^N$ 
are chosen to satisfy
\begin{align}
(-\alpha_{i,i} + A_{i,i} \varepsilon) a_i 
 & > \sum_{j \in I_i} \alpha_{i,j} a_j,
\label{eq:linear_constraint}
\end{align}
for every sufficiently small $\varepsilon > 0$, 
where $A_{i,i} = 0$ if $\alpha_{i,i} < 0$
and $A_{i,i} = 1$ if $\alpha_{i,i} = 0$.
Then, for every sufficiently small $\varepsilon > 0$
there exists $\mu_0 > 0$ such that
for every $\mu \in (0,\mu_0)$
the perturbed equation
\begin{align}
\frac{\mathrm{d} x_i}{\mathrm{d} t}
& = \alpha_{i,0} 
+ \alpha_{i,i} x_i 
+ \sum_{j \in I_i} \alpha_{i,j} x_j 
- A_{i,i} \varepsilon x_i,
\label{eq:i_linear_p2}
\end{align}
becomes chemical under the translation
$\mathbf{x} \to (\mathbf{x} - \mathbf{a}/\mu)$.
\end{lemma}

\begin{proof}
Translating via $\bar{\mathbf{x}} = \mathbf{x} + \mathbf{a}/\mu$
in~(\ref{eq:i_linear_p2}), one obtains the equation
\begin{align}
\frac{\mathrm{d} \bar{x}_i}{\mathrm{d} t}
& = \left[ - \frac{1}{\mu} 
\left((\alpha_{i,i} - A_{i,i} \varepsilon) a_i 
+ \sum_{j \in I_i} \alpha_{i,j} a_j \right) 
+  \alpha_{i,0} \right] 
+ (\alpha_{i,i} - A_{i,i} \varepsilon) \bar{x}_i 
+ \sum_{j \in I_i} \alpha_{i,j} \bar{x}_j,
\label{eq:case_1}
\end{align}
which, using~(\ref{eq:linear_constraint}), 
is chemical for every sufficiently small $\mu > 0$ 
by Definition~\ref{def:CDS}. 
\end{proof}

\noindent
\textbf{Remark}. Lemma~\ref{lemma:DS_linear_2}
implies that, if the $i$th equation 
from~(\ref{eq:DS_linear}) has $\alpha_{i,i} \le 0$, 
then, under the constraint~(\ref{eq:linear_constraint}),
one can take $p_i(\mathbf{x}; \varepsilon, \mu)
= - A_{i,i} \varepsilon x_i + \mu x_i r_i(\mathbf{x})/a_i$
in Theorem~\ref{theorem:DS_linear}
to obtain a new QCM achieving a CDS
with $M_2 = (M_{1}^{-} + N - 1)$ quadratics.
In the extreme case when this can be 
done for every equation of~(\ref{eq:DS_linear}), 
one obtains a CDS with only $M_2 = M_{1}^{-}$ quadratics.

We close this section by constructing a QCM
for a three-dimensional example;
for notational simplicity, we let $x = x_1$, $y = x_2$, 
$z = x_3$, and $a = a_1$, $b = a_2$ and $c = a_3$.

\begin{example}
\label{ex:linear}
Consider the linear \emph{DS}
\begin{align}
\frac{\mathrm{d} x}{\mathrm{d} t}
& = f_1(x) = \frac{1}{5} - \frac{57}{10} x, \nonumber \\
\frac{\mathrm{d} y}{\mathrm{d} t}
& = f_2(x,z) = x + z, \nonumber \\
\frac{\mathrm{d} z}{\mathrm{d} t}
& = f_3(y,z) = - y + \frac{z}{5},
\label{eq:Rossler_reflected_pre}
\end{align} 
with $M_1 = 5$ and $M_1^{-} = 1$.
Under the universal \emph{QCM} from 
\emph{Theorem}~$\ref{theorem:universal}$, 
$(\ref{eq:Rossler_reflected_pre})$ maps to
\begin{align}
\frac{\mathrm{d} \bar{x}}{\mathrm{d} t}
& = \frac{\mu}{a} \bar{x}
\left[\left(\frac{57}{10} \frac{a}{\mu} + \frac{1}{5} \right) 
- \frac{57}{10} \bar{x} \right], \nonumber \\
\frac{\mathrm{d} \bar{y}}{\mathrm{d} t}
& = \frac{\mu}{b} \bar{y}
\left[-\frac{(a + c)}{\mu} + \bar{x} + \bar{z} \right], 
\nonumber \\
\frac{\mathrm{d} \bar{z}}{\mathrm{d} t}
& = \frac{\mu}{c} \bar{z}
\left[\frac{1}{\mu} \left(b - \frac{c}{5} \right)- 
\bar{y} + \frac{\bar{z}}{5}\right],
\label{eq:R_example_1}
\end{align} 
which, by \emph{Definition}~$\ref{def:CDS_complexity}$, 
is a quadratic $(8,5)$ \emph{CDS}, 
i.e. it has $M_2 = 5$ quadratics.

Let us now apply the results from this section 
to construct a \emph{CDS} with only $M_2 = 2$ quadratics.
To this end, 
let us split~$(\ref{eq:Rossler_reflected_pre})$ as follows:
$f_1(x) = l_1(x) = 1/5 - 57 x/10$, 
$f_2(x,z) = l_2(x,z) = x + z$ and
$f_3(y,z) = r_3(y,z) = - y + z/5$.
Perturbing $l_1(x)$ and $l_2(x,z)$
according to \emph{Lemma}~$\ref{lemma:DS_linear_2}$,
and $r_3(y,z)$ according to 
\emph{Theorem}~$\ref{theorem:universal}$, one obtains 
\begin{align}
\frac{\mathrm{d} x}{\mathrm{d} t}
& = \frac{1}{5} - \frac{57}{10} x, \nonumber \\
\frac{\mathrm{d} y}{\mathrm{d} t}
& = x + z - \varepsilon y, \nonumber \\
\frac{\mathrm{d} z}{\mathrm{d} t}
& = - y + \frac{z}{5}
+ \frac{\mu}{c} z \left[- y + \frac{z}{5} \right].
\label{eq:Rossler_reflected_pre_p}
\end{align}
Under the translation
$(\bar{x},\bar{y},\bar{z}) = (x,y,z) + (a,b,c)/\mu$, 
the perturbed \emph{DS}~$(\ref{eq:Rossler_reflected_pre_p})$ becomes 
\begin{align}
\frac{\mathrm{d} \bar{x}}{\mathrm{d} t}
& = \left(\frac{57}{10} \frac{a}{\mu} + \frac{1}{5} \right) 
- \frac{57}{10} \bar{x}, \nonumber \\
\frac{\mathrm{d} \bar{y}}{\mathrm{d} t}
& = \frac{1}{\mu} 
\left(\varepsilon b - a - c \right) 
+ \bar{x} - \varepsilon \bar{y} + \bar{z}, \nonumber \\
\frac{\mathrm{d} \bar{z}}{\mathrm{d} t}
& = \frac{\mu}{c} \bar{z}
\left[\frac{1}{\mu} \left(b - \frac{c}{5} \right)- 
\bar{y} + \frac{\bar{z}}{5}\right], 
\label{eq:R_example_2} 
\end{align} 
which is a quadratic $(9,2)$ \emph{CDS} 
provided that, for every sufficiently small $\varepsilon > 0$,
we choose any $a, b, c > 0$ such that 
$b \ge (a + c)/\varepsilon$.
Let us note that one can also use the splitting
$f_3(y,z) = l_3(z) + r_3(y)$, and perturb $l_3(z) = z/5$
as per \emph{Lemma}~$\ref{lemma:DS_linear_1}$,
and $r_3(y) = -y$ via \emph{Theorem}~$\ref{theorem:universal}$;
this leads to the same quadratics in this particular example, 
because the third equation from~$(\ref{eq:Rossler_reflected_pre})$
contains no positive terms other than $z/5$.
\end{example}

\subsection{Quadratic DSs}
\label{sec:quadratic}
Consider DS~(\ref{eq:DS}) with $n = 2$, which can be written as
\begin{align}
\frac{\mathrm{d} x_i}{\mathrm{d} t}
= f_i(\mathbf{x}; 2) 
& = \alpha_{i,0} + \sum_{j = 1}^N \alpha_{i,j} x_j
+ \sum_{j = 1} \beta_{i,j} x_j^2
+ \sum_{\substack{j = 1, \\ j \ne i}}^N \gamma_{i,j} x_i x_j 
+ \sum_{\substack{(k,l) \in \pi(S_i), \\ k \ne l}} 
\delta_{i,k,l} x_k x_l,
\; \; \; i = 1,2,\ldots,N,
\label{eq:DS_quadratic}
\end{align}
where $\pi(S_i)$ denotes 
all the combinations of pairs from the set 
$S_i = \{1,2,\ldots,i-1,i+1,\ldots,N\}$.
In what follows, we assume that~(\ref{eq:DS}) with $n = 2$
has at least one non-zero quadratic term.
Let us consider a special case,
when the $i$th equation takes the form
\begin{align}
\frac{\mathrm{d} x_i}{\mathrm{d} t}
= q_i(\mathbf{x}) 
& = \beta_{i,i} x_i^2 
+ \sum_{j \in J_i} \beta_{i,j} x_j^2
+ \sum_{j \in J_i} \gamma_{i,j} x_i x_j 
+ \sum_{\substack{(k,l) \in \pi(J_i), \\ k \ne l}} \delta_{i,k,l} x_k x_l,
\label{eq:i_quadratic}
\end{align}
where $\beta_{i,i} \ge 0$, 
the set $J_i$ contains only indices
$j,k,l \ne i$ such that 
$ \beta_{i,j}, \delta_{i,k,l} \ge 0$
and $\gamma_{i,j} \le 0$, and $\pi(J_i)$ denotes 
all the combinations of pairs from $J_i$.

\begin{lemma}$($\textbf{\emph{Chemicality for~$(\ref{eq:i_quadratic})$}}$)$ 
\label{lemma:DS_quadratic}
Assume that parameters $\mathbf{a} \in \mathbb{R}_{>}^N$ 
can be chosen to satisfy
\begin{align}
\beta_{i,i} a_i & \ge \sum_{j \in J_i} (- \gamma_{i,j}) a_j, \nonumber \\
(- \gamma_{i,j}) a_i & \ge 2 \beta_{i,j} a_j
+ \sum_{l \in J_i, l \ne j} \delta_{i,j,l} a_l
\; \; \; \textrm{for all } j \in J_i.
\label{eq:quadratic_constraint}
\end{align}
Then, there exists $\mu_0 > 0$ such that
for every $\mu \in (0,\mu_0)$
equation~$(\ref{eq:i_quadratic})$
becomes chemical under the translation
$\mathbf{x} \to (\mathbf{x} - \mathbf{a}/\mu)$.
\end{lemma}

\begin{proof}
Translating via $\bar{\mathbf{x}} = \mathbf{x} + \mathbf{a}/\mu$
in~(\ref{eq:i_quadratic}), one obtains equation
 \begin{align}
\frac{\mathrm{d} \bar{x}_i}{\mathrm{d} t}
& =  \frac{1}{\mu^2} \left[\sum_{j \in J_i} \beta_{i,j} a_j^2
+ \sum_{\substack{(k,l) \in \pi(J_i), \\ k \ne l}}
\delta_{i,k,l} a_k a_l
+ a_i \left(\beta_{i,i}
a_i +  \sum_{j \in J_i} 
\gamma_{i,j} a_j \right)
\right]
- \frac{1}{\mu} \left(2 \beta_{i,i} a_i
+ \sum_{j \in J_i} \gamma_{i,j} a_j \right) \bar{x}_i \nonumber \\
& - \frac{1}{\mu} \sum_{j \in J_i} \left(
2 \beta_{i,j} a_j 
+ \gamma_{i,j} a_i
+ \sum_{l \in J_i, l \ne j} \delta_{i,j,l} a_l  \right) \bar{x}_j 
+ q_i(\bar{\mathbf{x}}),
\label{eq:1_perturbation_Q}
\end{align}
which, using~(\ref{eq:quadratic_constraint}), 
is chemical for every sufficiently small $\mu > 0$ 
by Definition~\ref{def:CDS}. 
\end{proof}

\noindent
\textbf{Remark}. If 
the coefficients $\beta_{i,i}, \beta_{i,j}, \delta_{i,j,k} > 0$
and $\gamma_{i,j} < 0$, then, for every
choice of $a_j > 0$ with $j \ne i$, 
conditions~(\ref{eq:quadratic_constraint})
are satisfied by taking any $a_i > 0$ sufficiently large. 
Therefore, conditions~(\ref{eq:quadratic_constraint}) 
can also be satisfied 
if one or more of these coefficients
are zero, provided that~(\ref{eq:i_quadratic})
is then perturbed by (some of) the underlying missing quadratics,
which we exploit in Theorem~\ref{theorem:DS_quadratic} below.

\noindent
\textbf{Remark}. Conditions~(\ref{eq:quadratic_constraint})
are sufficient for chemicality. More generally, one can 
replace the first inequality, which is linear, 
with the quadratic inequality obtained by demanding 
that the zero-degree monomial in~(\ref{eq:1_perturbation_Q})
is non-negative.

\noindent
\textbf{Remark}. In this section, we focus on 
reducing the number of cubics. More broadly,
to simultaneously decrease the number 
of linear and quadratic terms as well, 
one can redefine $q_i(\mathbf{x})$ by including in it
constant or suitable linear monomials.
Conditions~(\ref{eq:quadratic_constraint}), 
possibly with strict inequalities,
then remain sufficient for chemicality.
See Section~\ref{sec:two_quadratics} for some examples.

DS~(\ref{eq:DS_quadratic}) can be mapped under the universal 
QCM from Theorem~\ref{theorem:universal}
to a cubic CDS with $M_2$ cubics;
the same is also true for every DS obtained from~(\ref{eq:DS_quadratic})
via permutations or reflections of the dependent variables.
We now construct a QCM based on Lemma~\ref{lemma:DS_quadratic}
that is more efficient.

\begin{lemma} 
\label{lemma:1}
Every $N$-dimensional quadratic \emph{DS}
can be mapped via permutations and reflections 
to \emph{DS}~$(\ref{eq:DS_quadratic})$ 
whose first equation contains monomial
$m(\mathbf{x}) \in \{|\beta_{1,1}| x_1^2, - |\gamma_{1,2} |x_1 x_2, 
|\beta_{1,2}| x_2^2, |\delta_{1,2,3}| x_2 x_3\}$.
\end{lemma}

\begin{proof}
Every $N$-dimensional quadratic DS can be mapped
via a suitable permutation of the dependent variables
to DS~(\ref{eq:DS_quadratic}) with monomial 
$m(\mathbf{x}) \in \{\beta_{1,1} x_1^2, 
-\gamma_{1,2} x_1 x_2, \beta_{1,2} x_2^2, \delta_{1,2,3} x_2 x_3\}$
in the first equation. 
If $\beta_{1,1}, \beta_{1,2},\gamma_{1,2}, \delta_{1,2,3} > 0$, 
then the lemma follows.
Otherwise, proceed as follows:
if $m(\mathbf{x}) \in 
\{-|\beta_{1,1}| x_1^2,- |\beta_{1,2}| x_2^2\}$, then
apply the reflection $x_1 \to -x_1$ to obtain 
a DS with $m(\mathbf{x}) \in \{|\beta_{1,1}| x_1^2,|\beta_{1,2}| x_2^2\}$.
Similarly, if $m(\mathbf{x}) = |\gamma_{1,2}| x_1 x_2$
(respectively, $m(\mathbf{x}) = -|\delta_{1,2,3}| x_2 x_3$), then apply
$x_2 \to -x_2$ (respectively, $x_2 \to -x_2$ or $x_3 \to -x_3$)
to obtain $m(\mathbf{x}) = - |\gamma_{1,2}| x_1 x_2$
(respectively, $m(\mathbf{x}) = |\delta_{1,2,3}| x_2 x_3$).
\end{proof}

\begin{theorem}
\label{theorem:DS_quadratic}
Every $N$-dimensional quadratic \emph{DS}~$(\ref{eq:DS_quadratic})$, 
with exactly $M_2$ quadratic monomials,
can be mapped via permutations, reflections and a \emph{QCM} 
to an $N$-dimensional cubic \emph{CDS}
with $M_3 \le (M_2 - 1)$ cubic monomials.
\end{theorem}

\begin{proof}
By Lemma~\ref{lemma:1}, up to permutations and reflections, 
we assume without a loss of generality that 
DS~(\ref{eq:DS_quadratic}) has one of the four 
quadratic monomials $m(\mathbf{x})$
in the first equation. Defining remainders 
$r_1(\mathbf{x}) = f_1(\mathbf{x};2) - m(\mathbf{x})$,
and $r_i(\mathbf{x}) = f_i(\mathbf{x};2)$, $i = 2,3,\ldots,N$, 
let us split~(\ref{eq:DS_quadratic}) as follows:
\begin{align}
\frac{\mathrm{d} x_1}{\mathrm{d} t}
& = m(\mathbf{x}) + r_1(\mathbf{x}), 
\nonumber \\
\frac{\mathrm{d} x_i}{\mathrm{d} t}
& = r_i(\mathbf{x}), 
\; \; \; i = 2,3,\ldots,N.
\label{eq:DS_quadratic_splitting}
\end{align}
Consider the perturbed equations 
\begin{align}
\frac{\mathrm{d} x_1}{\mathrm{d} t}
& = m(\mathbf{x}) + r_1(\mathbf{x})
+ p(\mathbf{x}; \varepsilon) + \frac{\mu}{a_1} x_1 r_1(\mathbf{x}), 
\nonumber \\
\frac{\mathrm{d} x_i}{\mathrm{d} t}
& = r_i(\mathbf{x}) + \frac{\mu}{a_i} x_i r_i(\mathbf{x}),
\; \; \; i = 2,3,\ldots,N,
\label{eq:DS_quadratic_splitting_p}
\end{align}
where 
\begin{enumerate}
\item[(i)] if $m(\mathbf{x}) = |\beta_{1,1}| x_1^2$, 
then $p(\mathbf{x}; \varepsilon) \equiv 0$,
\item[(ii)] if $m(\mathbf{x}) = -|\gamma_{1,2}| x_1 x_2$, 
then $p(\mathbf{x}; \varepsilon) = \varepsilon x_1^2$
and we choose $a_1 \ge |\gamma_{1,2}| a_2/\varepsilon$,
\item[(iii)] if $m(\mathbf{x}) = |\beta_{1,2}| x_2^2$, 
then $p(\mathbf{x}; \varepsilon) = \varepsilon (x_1^2 - x_1 x_2)$
and we choose $a_1 \ge 2 |\beta_{1,2}| a_2/\varepsilon$,
\item[(iv)] if $m(\mathbf{x}) = |\delta_{1,2,3}| x_2 x_3$, 
then $p(\mathbf{x}; \varepsilon) = 
\varepsilon (x_1^2 - x_1 x_2 - x_1 x_3)$
and we choose $a_1 \ge \textrm{max} 
\{|\delta_{1,2,3}| a_2/\varepsilon,$ 
$|\delta_{1,2,3}| a_3/\varepsilon\}$.
\end{enumerate}
In each of these four cases, the function 
$q_1(\mathbf{x}) = m(\mathbf{x}) + p(\mathbf{x}; \varepsilon)$
is of the form given in~(\ref{eq:i_quadratic}), 
and the conditions~(\ref{eq:quadratic_constraint}) hold
for every sufficiently small $\varepsilon > 0$.
Therefore, for every sufficiently small $\varepsilon > 0$ 
and the stated choice of $\mathbf{a} \in \mathbb{R}_{>}^N$ 
there exists $\mu_0 > 0$ such that
for every $\mu \in (0,\mu_0)$
the perturbed equations~(\ref{eq:DS_quadratic_splitting_p})
become chemical under the translation $\mathbf{x}
\to \mathbf{x} - \mathbf{a}/\mu$ by Lemma~\ref{lemma:DS_quadratic}, 
Theorem~\ref{theorem:universal} and Lemma~\ref{lemma:splitting}.
By Definition~\ref{def:QCM}, the perturbations 
$p_1(\mathbf{x}; \varepsilon,\mu) 
= p(\mathbf{x}; \varepsilon) + \mu x_1 r_1(\mathbf{x})/a_1$, 
$p_i(\mathbf{x}; \mu) = \mu x_i r_i(\mathbf{x})/a_i$, 
$i = 2,3,\ldots,N$, 
and the translation vector $\mathbf{a}/\mu$,
then induce a QCM for DS~(\ref{eq:DS_quadratic}).
This QCM maps one of the $M_2$ quadratic monomials
into quadratics, and not a cubic; 
hence, the resulting CDS has 
exactly $M_3 = (M_2 - 1)$ cubics.
\end{proof}

\noindent
\textbf{Remark}. Depending on the precise 
structure of~(\ref{eq:DS_quadratic}),
further reduction of cubic terms may be possible 
via a more refined QCM: by including into $q_1(\mathbf{x})$ 
more quadratic terms from 
$r_1(\mathbf{x})$, and by applying such splittings
to the other equations as well. See
Section~\ref{sec:two_quadratics} for an example.

\begin{example}
Consider R\"ossler system~$(\ref{eq:Rossler})$.
Under the universal \emph{QCM} from 
\emph{Theorem}~$\ref{theorem:universal}$, 
$(\ref{eq:Rossler})$ maps to a cubic \emph{CDS}.
Let us now apply the results from this section
to construct a \emph{QCM} that maps 
R\"ossler system into a quadratic \emph{CDS}.
To this end, as per \emph{Lemma~\ref{lemma:1}}, 
we apply the reflection $y \to - y$, thus obtaining
\begin{align}
\frac{\mathrm{d} x}{\mathrm{d} t}
& = f_1(x,y) = \frac{1}{5} - \frac{57}{10} x - x y, \nonumber \\
\frac{\mathrm{d} y}{\mathrm{d} t}
& = f_2(x,z) = x + z, \nonumber \\
\frac{\mathrm{d} z}{\mathrm{d} t}
& = f_3(y,z) = - y + \frac{z}{5}.
\label{eq:Rossler_reflected}
\end{align}
Let us split this \emph{DS} as follows:
$f_1(x,y) = m(x,y)  + l_1(x)$
with $m(x,y) = - x y$,
$f_2(x,z) = l_2(z) + r_2(x)$
with $l_2(z) = z$, 
and $f_3(y,z) = r_3(y,z)$.
Perturbing then $m(x,y)$
according to \emph{Theorem~\ref{theorem:DS_quadratic}(ii)}, 
$l_1(x)$ and $l_2(z)$ according to \emph{Lemma~\ref{lemma:DS_linear_2}},
and $r_2(x)$ and $r_3(y,z)$ 
via \emph{Theorem~\ref{theorem:universal}}, 
one obtains
\begin{align}
\frac{\mathrm{d} x}{\mathrm{d} t}
& = \frac{1}{5} - \frac{57}{10} x - x y
+ \varepsilon x^2, \nonumber \\
\frac{\mathrm{d} y}{\mathrm{d} t}
& = x + z - \varepsilon y + \frac{\mu}{b} x y, \nonumber \\
\frac{\mathrm{d} z}{\mathrm{d} t}
& = - y + \frac{z}{5} + \frac{\mu}{c} 
z \left(-y + \frac{z}{5} \right).
\label{eq:Rossler_reflected_perturbed}
\end{align}
Under the translation
$(\bar{x},\bar{y},\bar{z}) = (x,y,z) + (a,b,c)/\mu$, 
the perturbed \emph{DS}~$(\ref{eq:Rossler_reflected_perturbed})$ 
becomes 
\begin{align}
\frac{\mathrm{d} \bar{x}}{\mathrm{d} t}
& = \left( \frac{a}{\mu^2} (\varepsilon a - b)
+ \frac{57 a}{10 \mu} + \frac{1}{5}  \right)
+ \left( \frac{1}{\mu} (b - 2 \varepsilon a) - \frac{57}{10} \right) \bar{x}
+ \frac{a}{\mu} \bar{y} + \varepsilon \bar{x}^2 - \bar{x} \bar{y}, 
\nonumber \\
\frac{\mathrm{d} \bar{y}}{\mathrm{d} t}
& = \frac{1}{\mu} \left(\varepsilon b - c \right)
- \left(\frac{a}{b} + \varepsilon \right) \bar{y} + \bar{z}
+ \frac{\mu}{b} \bar{x} \bar{y}, \nonumber \\
\frac{\mathrm{d} \bar{z}}{\mathrm{d} t}
& = \left(\frac{b}{c} - \frac{1}{5}  \right) \bar{z}
+ \frac{\mu}{5 c} \bar{z}^2 
- \frac{\mu}{c} \bar{y} \bar{z}.
\label{eq:Rossler_reflected_perturbed_translated}
\end{align}
For every $\varepsilon > 0$,
choosing any $a,b,c > 0$ such that 
$a \ge b/\varepsilon$ and $b \ge c/\varepsilon$,
it follows that, for every sufficiently small $\mu > 0$,
$(\ref{eq:Rossler_reflected_perturbed_translated})$
is a quadratic $(12,5)$ \emph{CDS}.
\end{example}

\section{Applications to chaos}
\label{sec:applications}
In this section, we apply the theory developed
in Section~\ref{sec:DS_to_CDS} to quadratic
DSs with up to two quadratic monomials
to study the capacity of chemical and biological systems
for chaos. In particular, prioritizing reduction in the number of 
the highest-degree monomials and reactions,
we construct three relatively simple CDSs 
displaying chaotic attractors with various properties.

\subsection{Chaotic DSs with one quadratic monomial}
\label{sec:one_quadratic}
By Theorem~\ref{theorem:DS_quadratic}, every quadratic DS,
with only one quadratic term, can be mapped to a
\emph{quadratic} CDS; in the context of chaos, we then 
have the following result.

\begin{theorem} 
\label{theorem:one_quadratic}
Assume that $N$-dimensional quadratic \emph{DS},
with only one quadratic monomial, 
has a robust bounded invariant set in 
$\mathbb{R}^N$ with robust properties 
$\mathcal{P}$ that are invariant under permutations, reflections 
and translations of the dependent variables. 
Then, there exists an $N$-dimensional quadratic \emph{CDS}
with a bounded invariant set in $\mathbb{R}_{>}^N$ with 
properties $\mathcal{P}$. 
\end{theorem}

\begin{proof}
By Theorem~\ref{theorem:DS_quadratic}, every 
$N$-dimensional quadratic DS,
with only one quadratic monomial, can be mapped
to an $N$-dimensional quadratic CDS
via permutations, reflections and a QCM,
the latter of which consists purely of 
smooth perturbations of the vector field
and translations of the dependent variables
by Definition~\ref{def:QCM}.
By Definition~\ref{def:robustness},
robust invariant sets and robust properties
are preserved under such perturbations,
implying the statement of the theorem.
\end{proof}

Theorem~\ref{theorem:one_quadratic} implies that,
assuming robust chaos, the R\"ossler system~(\ref{eq:Rossler}),
and the Sprott systems F--S~\cite{Sprott}[Table 1],
can all be mapped to quadratic CDSs with chaos.

\begin{figure}[!htbp]
\vskip -2.0cm
\leftline{
\hskip 0.0cm (a) Original DS~(\ref{eq:Rossler_reflected})
\hskip  2.4cm (d) Perturbed DS~(\ref{eq:Rossler_reflected_perturbed}) 
\hskip  2.5cm (g) CDS~(\ref{eq:CDS_Rossler_p})}
\vskip 0.0cm
\centerline{
\hskip 0.0cm
\includegraphics[width=0.35\columnwidth]{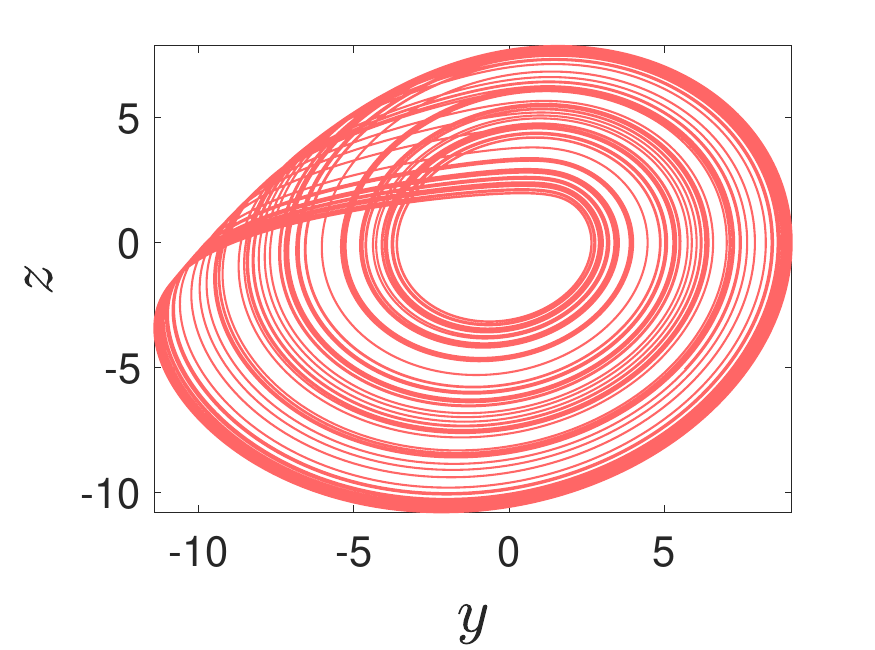}
\hskip 0.2cm
\includegraphics[width=0.35\columnwidth]{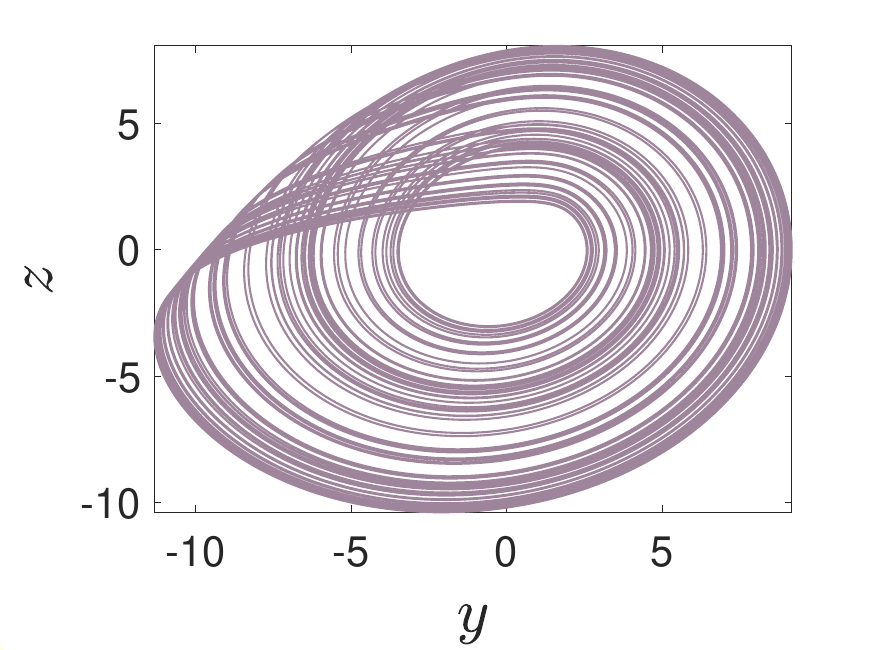}
\hskip 0.2cm
\includegraphics[width=0.35\columnwidth]{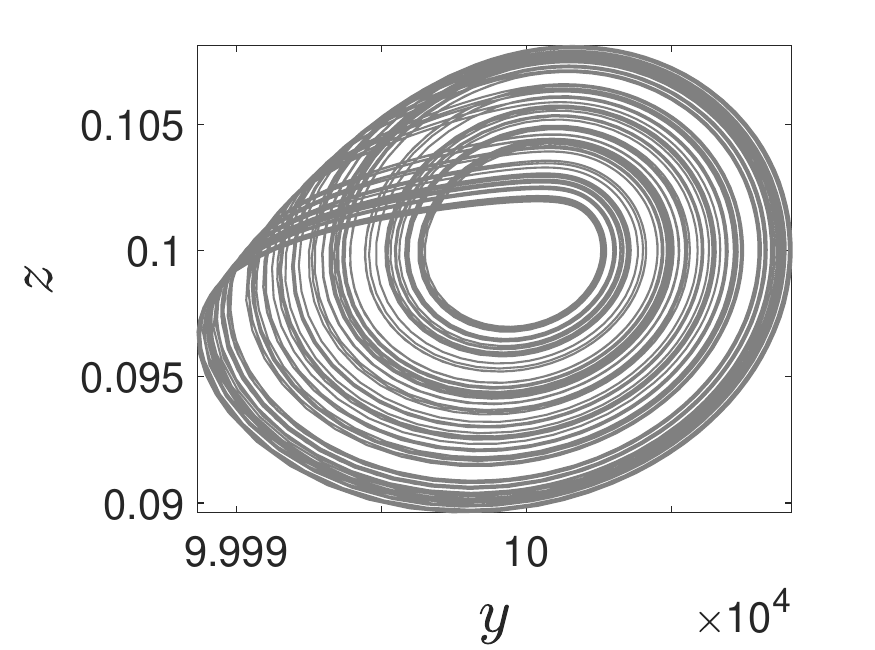}
}
\vskip 0.0cm
\leftline{
\hskip 0.0cm (b) 
\hskip  5.5cm (e)
\hskip  5.5cm (h)}
\centerline{
\hskip 0.0cm
\includegraphics[width=0.35\columnwidth]{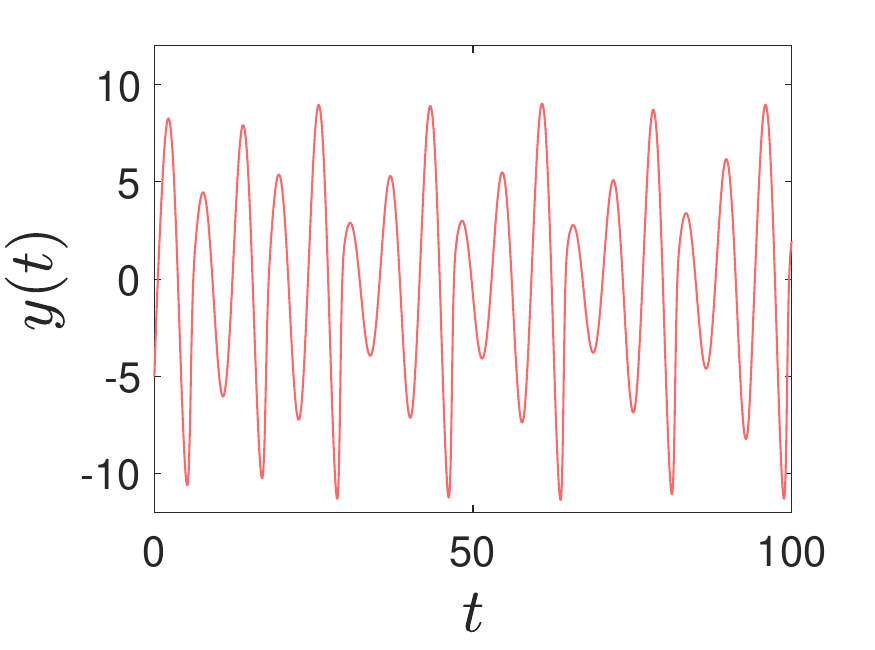}
\hskip 0.2cm
\includegraphics[width=0.35\columnwidth]{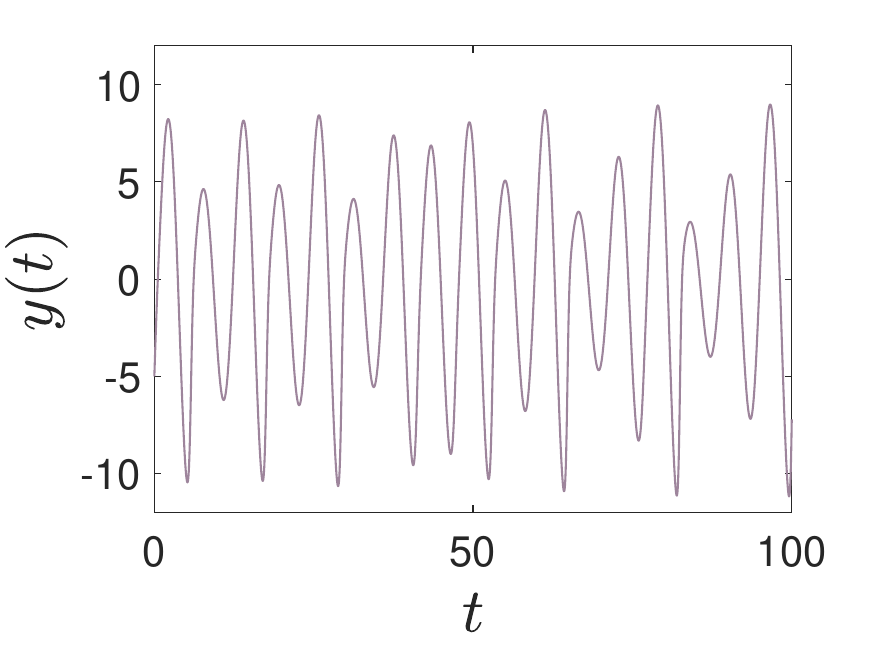}
\hskip 0.2cm
\includegraphics[width=0.35\columnwidth]{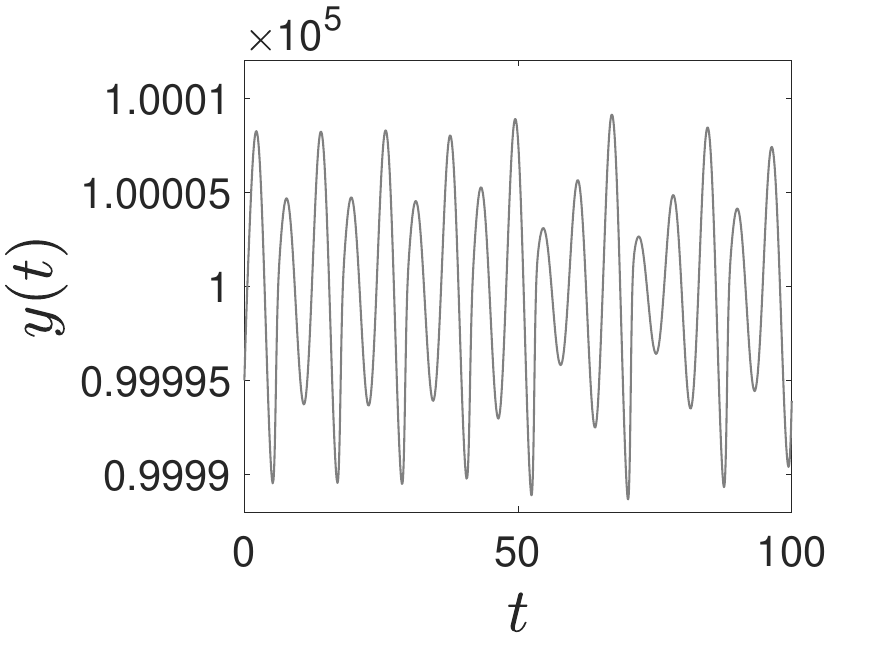}
}
\vskip 0.0cm
\leftline{
\hskip 0.0cm (c) 
\hskip  5.5cm (f)
\hskip  5.5cm (i)}
\centerline{
\hskip 0.0cm
\includegraphics[width=0.35\columnwidth]{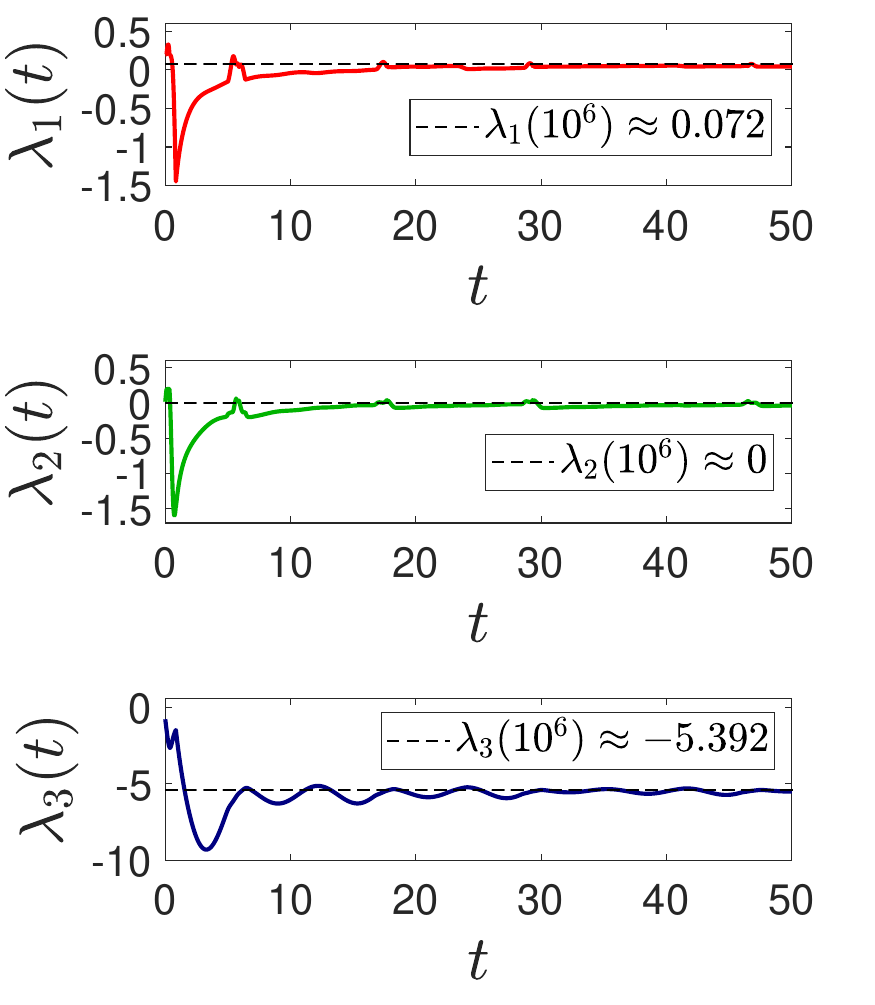}
\hskip 0.2cm
\includegraphics[width=0.35\columnwidth]{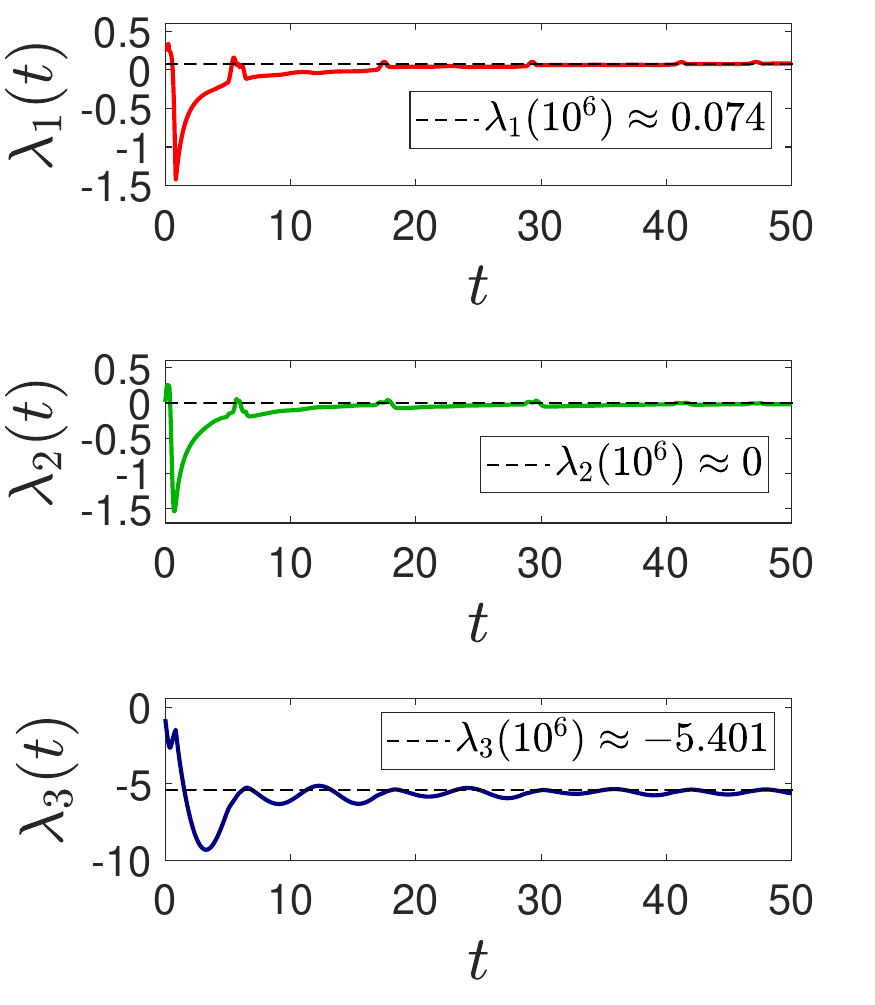}
\hskip 0.2cm
\includegraphics[width=0.35\columnwidth]{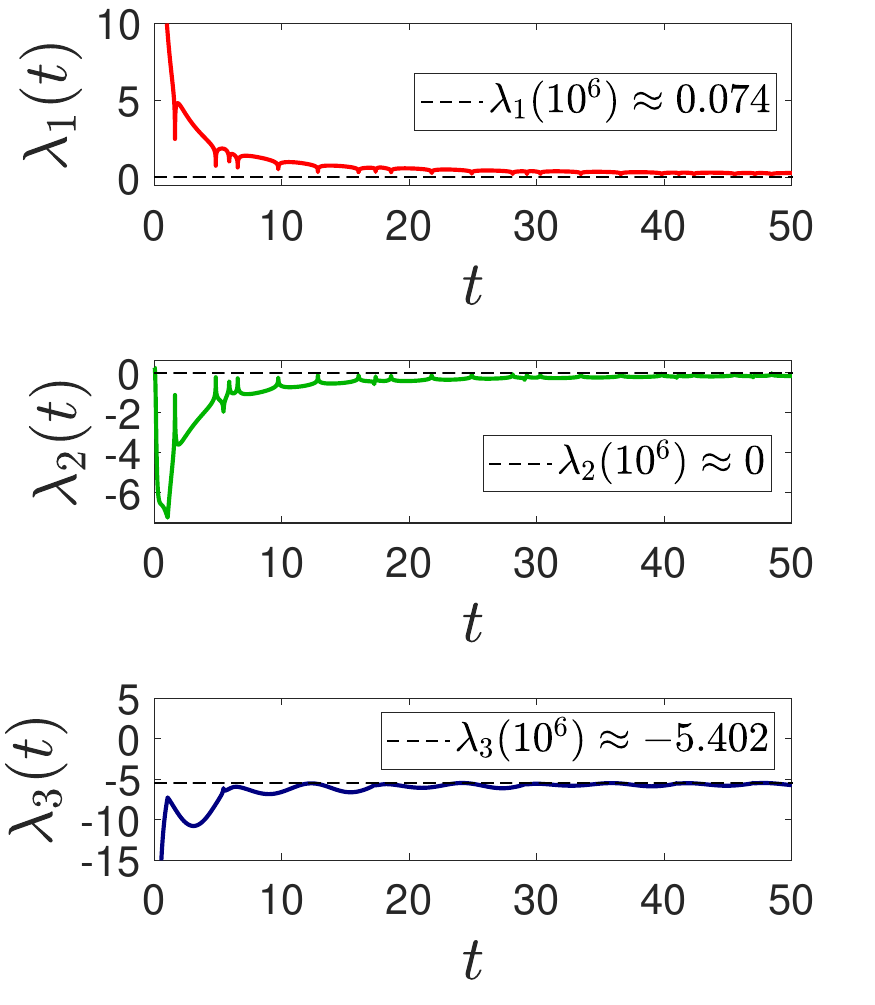}
}
\vskip -0.3cm
\caption{\it{\emph{Quadratic $(11,5)$ chemical R\"ossler system.} 
Panels~\emph{(a)} and~\emph{(b)} display 
respectively the $(y,z)$-
and $(t,y)$-space for the reflected 
R\"ossler system~$(\ref{eq:Rossler_reflected})$
with initial condition $(x_0,y_0,z_0) = (5,-5,5)$; 
panel~\emph{(c)} shows time-evolution of the 
three \emph{LCE}s for this solution, along with the values
at $t = 10^6$ as dashed black lines. 
Analogous plots are shown in panels~\emph{(d)}--\emph{(f)} 
for the perturbed 
R\"ossler system~$(\ref{eq:Rossler_reflected_perturbed})$,
and in~\emph{(g)}--\emph{(i)} 
for its translated and rescaled version~$(\ref{eq:CDS_Rossler})$
with the translated and rescaled initial condition
$x_0 = \varepsilon \mu (5 + (\varepsilon^2 \mu)^{-1})$,
$y_0 = -5 + (\varepsilon \mu)^{-1}$,
$z_0 = 5 \varepsilon (5 - \varepsilon)^{-1} (5 + \mu^{-1})$. 
The parameters are fixed to
$(a,b,c) = (\varepsilon^{-2},\varepsilon^{-1},1)$ and
$(\varepsilon,\mu) = (10^{-3}, 10^{-2})$.}} 
\label{fig:CDS_R}
\end{figure}

\begin{theorem}$($\textbf{\emph{Chemical R\"ossler system}}$)$ 
\label{theorem:Rossler}
Consider three-dimensional quadratic $(11,5)$ \emph{CDS} 
\begin{align}
\frac{\mathrm{d} x}{\mathrm{d} t}
& = \alpha_1 - \alpha_2 x + \alpha_3 y 
+ \alpha_4 x^2 - \alpha_5 x y, 
\nonumber \\
\frac{\mathrm{d} y}{\mathrm{d} t}
& = - \alpha_6 y + \alpha_7 z + \alpha_5 x y, \nonumber \\
\frac{\mathrm{d} z}{\mathrm{d} t}
& = \alpha_7 z + \alpha_8 z^2 - \alpha_9 y z.
\label{eq:CDS_Rossler}
\end{align}
 with a $(9,4)$ \emph{CRN} 
\begin{align}
\varnothing & \xrightarrow[]{\alpha_1} X, \; \; \; 
X \xrightarrow[]{\alpha_2} \varnothing,  \; \; \; 
Y \xrightarrow[]{\alpha_3} X + Y, \; \; \; 
2 X \xrightarrow[]{\alpha_4} 3 X, \; \; \; 
X + Y \xrightarrow[]{\alpha_5} 2 Y,  \nonumber \\
Y & \xrightarrow[]{\alpha_6} \varnothing, \; \; \; 
Z \xrightarrow[]{\alpha_7} Y + 2 Z, \; \; \; 
2 Z \xrightarrow[]{\alpha_8} 3 Z, \; \; \; 
Y + Z \xrightarrow[]{\alpha_9} Y,
\label{eq:CRN_Rossler}
\end{align}
and parameters
\begin{align}
\alpha_1 & = \frac{57}{10 \varepsilon} + \frac{\varepsilon \mu}{5}, 
\; \; 
\alpha_2 =  \frac{1}{\varepsilon \mu} + \frac{57}{10} , \; \; 
\alpha_3 = \frac{1}{\varepsilon}, \; \; 
\alpha_4 =  \frac{1}{\mu}, \; \;
\alpha_5 = 1,\nonumber \\
\alpha_6 & = \frac{1}{\varepsilon} + \varepsilon, \; \; 
\alpha_7 = \frac{1}{\varepsilon} - \frac{1}{5} , \; \; 
\alpha_8 = \frac{\mu}{5 \varepsilon} - \frac{\mu}{25}, \; \; 
\alpha_9 = \mu.
\label{eq:parameters_Rossler}
\end{align}
Assume that \emph{DS}~$(\ref{eq:Rossler})$
has a robust chaotic attractor in $\mathbb{R}^3$. 
Then, \emph{CDS}~$(\ref{eq:CDS_Rossler})$ has a 
chaotic attractor in $\mathbb{R}_{>}^3$
for every sufficiently small $\varepsilon > 0$ and $\mu > 0$. 
\end{theorem}

\begin{proof}
By assumption, (\ref{eq:Rossler}) has a robust chaotic attractor;
therefore, by Definition~\ref{def:robustness}, 
there exists $\varepsilon_0 > 0$ such that 
for every $\varepsilon \in (0,\varepsilon_0)$,
and for every $a,b,c > 0$,
there exists $\mu_0 > 0$ such that
for every $\mu \in (0,\mu_0)$
the perturbed system~(\ref{eq:Rossler_reflected_perturbed})
also has a robust chaotic attractor. 
Being a translation of~(\ref{eq:Rossler_reflected_perturbed}), 
DS~(\ref{eq:Rossler_reflected_perturbed_translated})
also has a robust chaotic attractor, which is 
in $\mathbb{R}_{>}^3$ for every sufficiently small $\mu$.

Choosing any $a,b,c > 0$ such that 
$a \ge b/\varepsilon$ and $b \ge c/\varepsilon$
ensures that~(\ref{eq:Rossler_reflected_perturbed_translated}) 
is chemical. To reduce the number of the induced chemical 
reactions, let us choose $b = c/\varepsilon$ 
and e.g. $a = c/\varepsilon^2$ and $c = 1$,
under which choice the CDS becomes
\begin{align}
\frac{\mathrm{d} \bar{x}}{\mathrm{d} t}
& = \left(\frac{57}{10 \varepsilon^2 \mu} + \frac{1}{5}  \right)
- \left( \frac{1}{\varepsilon \mu} + \frac{57}{10} \right) \bar{x}
+ \frac{1}{\varepsilon^2 \mu} \bar{y} 
+ \varepsilon \bar{x}^2 - \bar{x} \bar{y}, 
\nonumber \\
\frac{\mathrm{d} \bar{y}}{\mathrm{d} t}
& = - \left(\frac{1}{\varepsilon} + \varepsilon \right) \bar{y} + \bar{z}
+ \varepsilon \mu \bar{x} \bar{y}, \nonumber \\
\frac{\mathrm{d} \bar{z}}{\mathrm{d} t}
& = \left(\frac{1}{\varepsilon} - \frac{1}{5}  \right) \bar{z}
+ \frac{\mu}{5} \bar{z}^2 
- \mu \bar{y} \bar{z}.
\label{eq:Rossler_c_prescale}
\end{align}
To further reduce the number of reactions, 
let us rescale via $\bar{x} \to \bar{x}/(\varepsilon \mu)$
and $\bar{z} \to (1/\varepsilon - 1/5) \bar{z}$, 
which ensures that the terms
$\bar{x} \bar{y}$ in the first 
and second equations, and the terms 
$z$ in the second and third equations, 
are multiplied, up to sign, by the same coefficients,
thus allowing for fusion of the underlying reactions
(see Appendix~\ref{app:CRN}).
Removing the bars, (\ref{eq:Rossler_c_prescale}) 
then becomes the CDS~(\ref{eq:CDS_Rossler}).
\end{proof}

\begin{example}
\label{ex:chemical_Rossler}
Under the parameter choice 
$\varepsilon = 10^{-3}$ and $\mu = 10^{-2}$,
\emph{CDS}~$(\ref{eq:CDS_Rossler})$ becomes
\begin{align}
\frac{\mathrm{d} x}{\mathrm{d} t}
& = 5700.000002 - 100005.7 x + 1000 y
+ 100 x^2 - x y, \nonumber \\
\frac{\mathrm{d} y}{\mathrm{d} t}
& = - 1000.001 y + 999.8 z + x y, \nonumber \\
\frac{\mathrm{d} z}{\mathrm{d} t}
& = 999.8 z + 1.9996 z^2 - 0.01 y z.
\label{eq:CDS_Rossler_p}
\end{align}
In \emph{Figure}~$\ref{fig:CDS_R}(a)$, we display the numerically
observed chaotic attractor of the (reflected) 
R\"ossler system~$(\ref{eq:Rossler_reflected})$
in the $(y,z)$-space, while in \emph{Figure}~$\ref{fig:CDS_R}(b)$
the corresponding trajectory in the $(t,y)$-space.
Furthermore, in \emph{Figure}~$\ref{fig:CDS_R}(c)$, we show
the three finite-time \emph{LCE}s over the time-interval 
$t \in [0,50]$, along with the values
at $t = 10^6$ shown as the dashed lines. 
The positive \emph{LCE} is consistent with chaos; 
see \emph{Appendix~\ref{app:LCE}} for more details
about \emph{LCE}s, including a numerical method 
used in this paper to approximate them.
Analogous plots are shown in \emph{Figure}~$\ref{fig:CDS_R}(d)$--$(f)$ 
for the perturbed R\"ossler system~$(\ref{eq:Rossler_reflected_perturbed})$
with $b = \varepsilon^{-1}$ and $c = 1$;
parameters $\varepsilon = 10^{-3}$ and $\mu = 10^{-2}$
are chosen so that there is a good match with the original attractor.
In particular, note that the \emph{LCE}s of~$(\ref{eq:Rossler_reflected_perturbed})$ are relatively close
to of those of~$(\ref{eq:Rossler_reflected})$, 
numerically indicating robustness of the underlying chaotic attractor.
Finally, in \emph{Figure}~$\ref{fig:CDS_R}(g)$--$(i)$, we show
analogous plots for the perturbed, translated and rescaled
\emph{CDS}~$(\ref{eq:CDS_Rossler})$, 
which is under the chosen parameter values equivalent 
to~$(\ref{eq:CDS_Rossler_p})$.
Note that, since infinite-time \emph{LCE}s are 
invariant under affine transformations, 
systems~$(\ref{eq:Rossler_reflected_perturbed})$ 
and~$(\ref{eq:CDS_Rossler})$ have 
identical \emph{LCE}s in the long-run.
\end{example}

\subsubsection{One-wing chaos}
The $(11,5)$ chemical R\"ossler system~(\ref{eq:CDS_Rossler})
with $(9,4)$ CRN~(\ref{eq:CRN_Rossler})
has one fewer quadratic than the $(9,6)$ Willamowski–R\"ossler 
system~(\ref{eq:Willamowski_Rossler_CDS}) 
with $(7,4)$ CRN~(\ref{eq:Rossler_CRN});
however, they contain the same number of quadratic reactions. 
Let us now construct an even simpler CDS with one-wing chaos.
To this end, consider the quadratic DS
\begin{align}
\frac{\mathrm{d} x}{\mathrm{d} t} & = - y + x^2, \nonumber \\
\frac{\mathrm{d} y}{\mathrm{d} t} & = \frac{27}{10} x + z, \nonumber \\
\frac{\mathrm{d} z}{\mathrm{d} t} & = x + y,
\label{eq:DS_1}
\end{align}
obtained via permutation $x \to y$
in Sprott system P~\cite{Sprott}[Table 1].
We display the numerically observed chaotic attractor
of~(\ref{eq:DS_1}) in Figure~\ref{fig:CDS_1}(a)--(b),
and the LCEs in Figure~\ref{fig:CDS_1}(c).

\begin{figure}[!htbp]
\vskip -0.0cm
\leftline{\hskip 
0.1cm (a) DS~(\ref{eq:DS_1}) \hskip  
4.1cm (b) \hskip  
5.3cm (c)}
\centerline{
\hskip 0.0cm
\includegraphics[width=0.35\columnwidth]{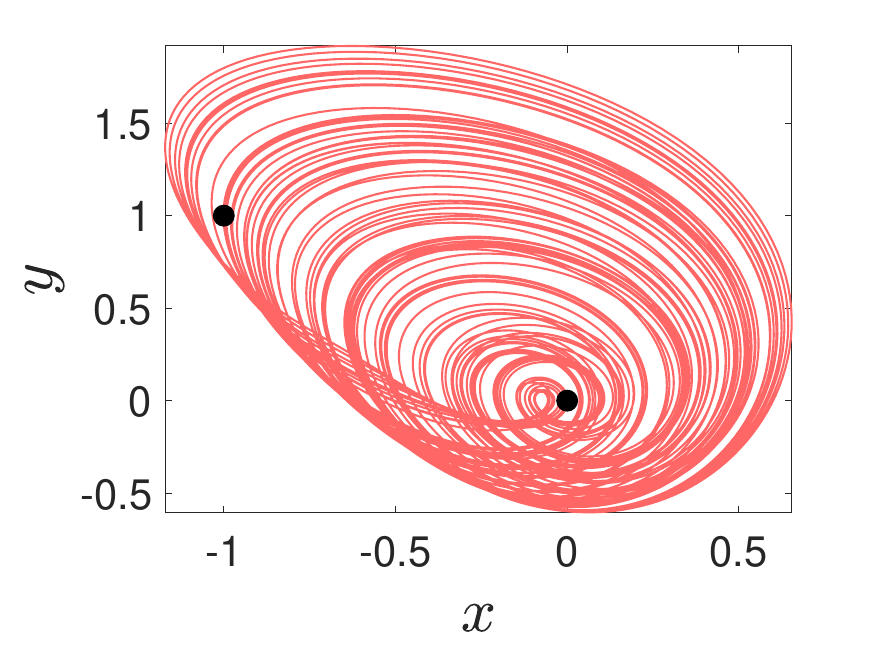}
\hskip 0.2cm
\includegraphics[width=0.35\columnwidth]{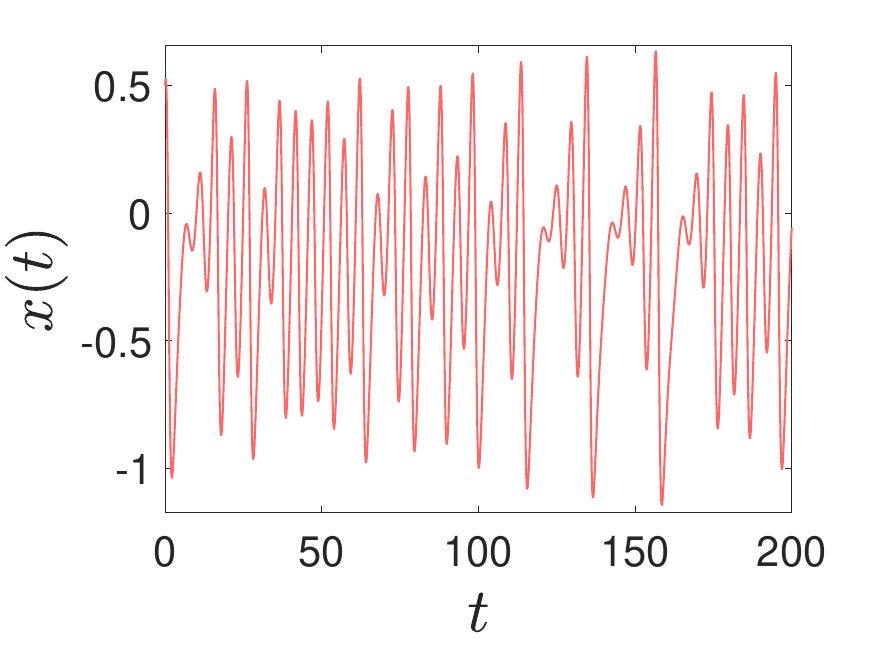}
\hskip 0.2cm
\includegraphics[width=0.35\columnwidth]{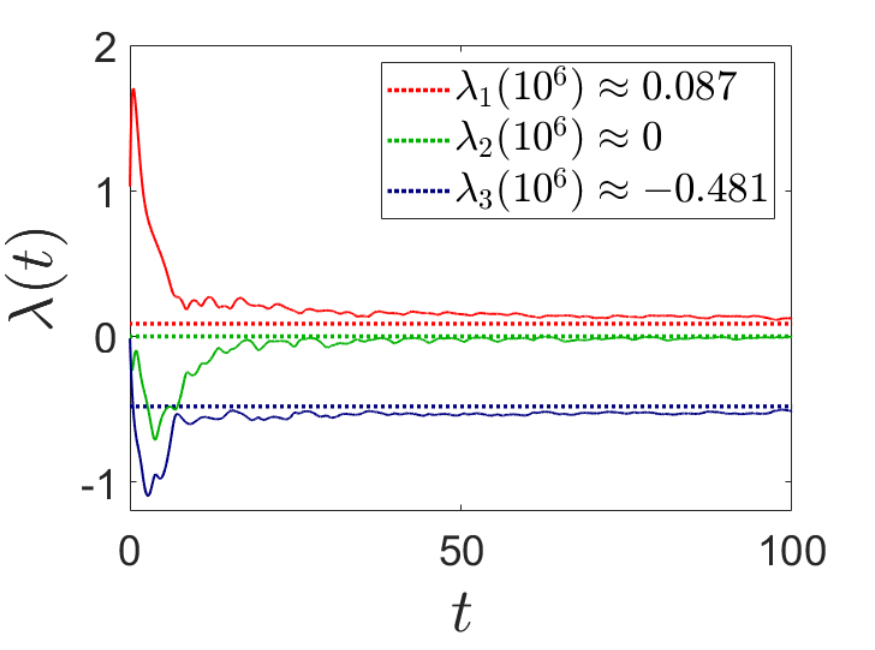}
}
\leftline{\hskip 
0.1cm (d) CDS~(\ref{eq:CDS_1}) \hskip  
3.9cm (e) \hskip  
5.3cm (f)}
\vskip 0.2cm
\centerline{
\hskip 0.0cm
\includegraphics[width=0.35\columnwidth]{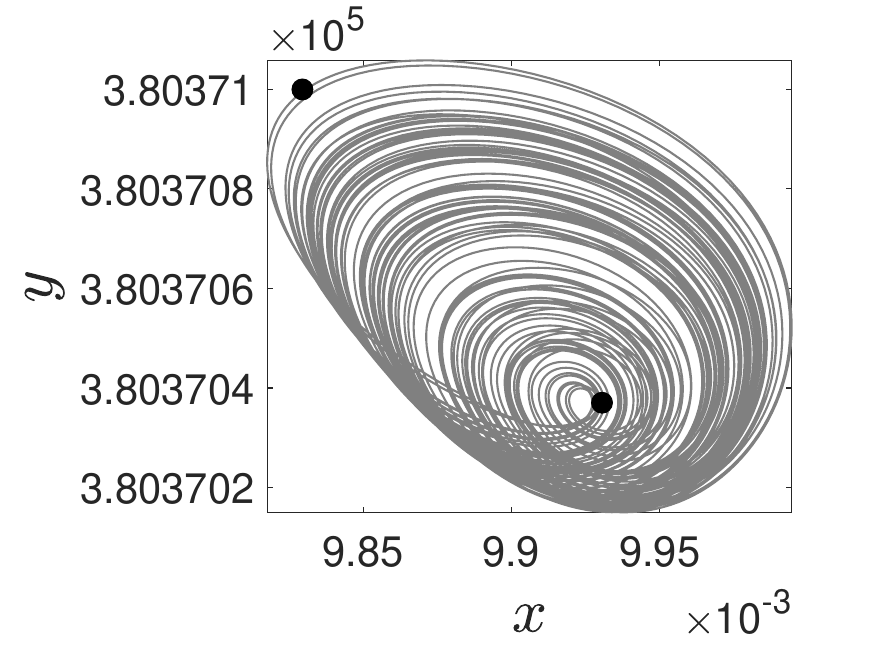}
\hskip 0.2cm
\includegraphics[width=0.35\columnwidth]{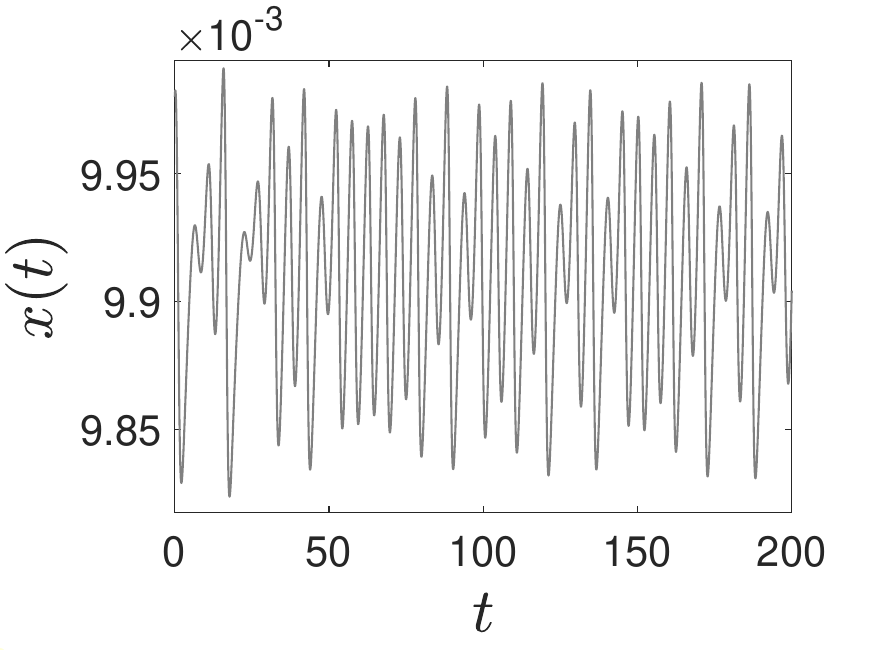}
\hskip 0.2cm
\includegraphics[width=0.35\columnwidth]{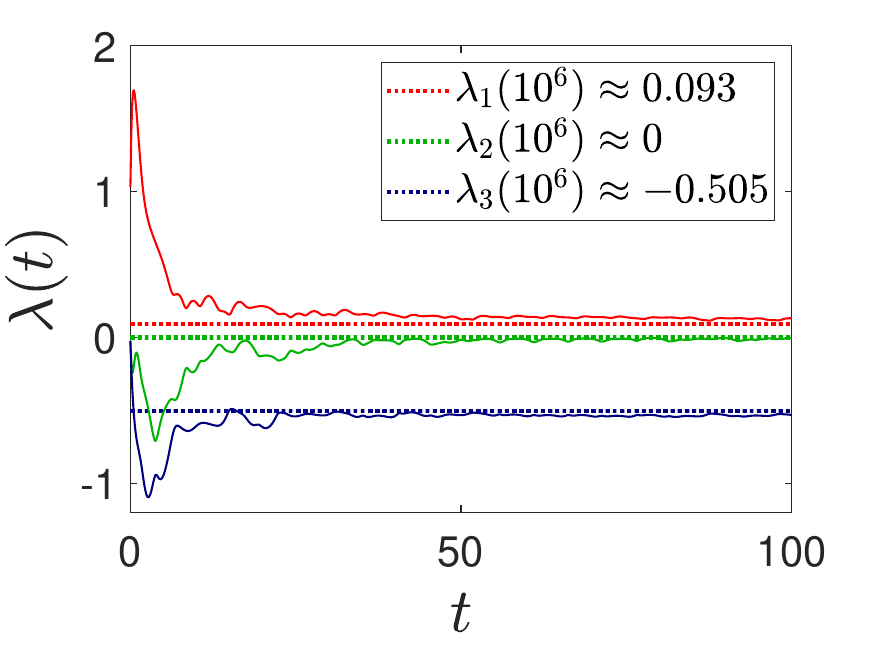}
}
\vskip -0.2cm
\caption{\it{\emph{Quadratic $(10,3)$ CDS with one-wing chaos.} 
Panels \emph{(a)}--\emph{(b)} display 
respectively the $(x,y)$-
and $(t,x)$-space for \emph{DS}~$(\ref{eq:DS_1})$
with initial condition $(x_0,y_0,z_0)
= (0.5,0,0)$; the two equilibria 
are shown as black dots in \emph{(a)}.
Panel \emph{(c)} shows
the three finite-time \emph{LCE}s, 
along with the values
at $t = 10^6$ shown as the dashed lines.
Panels \emph{(d)}--\emph{(e)} show analogous plots 
for \emph{CDS}~$(\ref{eq:CDS_1})$
with $(\varepsilon,\mu) = (10^{-2},10^{-2})$
and the translated and rescaled initial condition 
$x_0 = |\varepsilon^{-2} + 2.7 \varepsilon^{-1} - 2 \mu^{-1}|^{-1} 
(0.5 + \mu^{-1})$, 
$y_0 = (10/27) ((\varepsilon^2 \mu)^{-1} + 2.7 (\varepsilon \mu)^{-1})$,
$z_0 = (\varepsilon \mu)^{-1}$. 
Panel \emph{(f)} shows the \emph{LCE}s 
for the perturbed \emph{DS}~$(\ref{eq:1_perturbed})$
with $a = 1$ and $c = \varepsilon^{-1}$, 
which are in the long-run identical
to the long-time \emph{LCE}s of~$(\ref{eq:CDS_1})$.}} 
\label{fig:CDS_1}
\end{figure}

\begin{theorem}$($\textbf{\emph{One-wing chaos}}$)$ 
\label{theorem:one_wing}
Consider three-dimensional quadratic $(10,3)$ \emph{CDS}
\begin{align}
\frac{\mathrm{d} x}{\mathrm{d} t} 
& = \alpha_1 + \alpha_2 x + \alpha_3 x^2
- \alpha_4 x y, \nonumber \\
\frac{\mathrm{d} y}{\mathrm{d} t} 
& = |\alpha_2| x - \alpha_5 y + \alpha_6 z, 
\nonumber \\
\frac{\mathrm{d} z}{\mathrm{d} t} 
& = |\alpha_2| x - \alpha_7 z + \alpha_8 y z, 
\label{eq:CDS_1}
\end{align}
with $(8,3)$ \emph{CRN}
\begin{align}
\varnothing  & \xrightarrow[]{\alpha_1} X, \; \; \; 
X \xrightarrow[]{|\alpha_2|} X + 
\textrm{\emph{sign}}(\alpha_2) X + Y + Z,  \; \; \; 
2 X \xrightarrow[]{\alpha_3} 3 X, \; \; \; 
X + Y \xrightarrow[]{\alpha_4} Y,\nonumber \\
Y & \xrightarrow[]{\alpha_5} \varnothing, \; \; \; 
Z \xrightarrow[]{\alpha_6} Y + Z, \; \; \; 
Z \xrightarrow[]{\alpha_7} \varnothing, \; \; \; 
Y + Z \xrightarrow[]{\alpha_8} Y + 2 Z,
\label{eq:CRN_1}
\end{align}
and parameters
\begin{align}
\alpha_1 & = \frac{1}{\mu^2 |\alpha_2|}, \; \; 
\alpha_2 = \frac{1}{\varepsilon^2} 
+ \frac{27}{10 \varepsilon} - \frac{2}{\mu}, \; \; 
\alpha_3 = |\alpha_2|, \; \; \alpha_4 = \frac{27}{10} \mu, \nonumber \\
\alpha_5 & = \varepsilon, \; \; 
\alpha_6 = \frac{10}{27}, \; \; 
\alpha_7 = \frac{1}{\varepsilon} + \frac{27}{10} + \varepsilon, \; \; 
\alpha_8 =  \frac{27}{10} \varepsilon \mu .
\label{eq:parameters_1}
\end{align}
Assume that \emph{DS}~$(\ref{eq:DS_1})$
has a robust one-wing chaotic attractor in $\mathbb{R}^3$. 
Then, \emph{CDS}~$(\ref{eq:CDS_1})$ has a 
one-wing chaotic attractor in $\mathbb{R}_{>}^3$
for every sufficiently small $\varepsilon > 0$ and $\mu > 0$. 
\end{theorem}

\begin{proof}
Let us split DS~(\ref{eq:DS_1}) as follows:
$f_1(x,y) = - y + x^2 = m(x) + r_1(y)$ with $m(x) = x^2$, 
$f_2(x,z) = 27 x/10 + z = l_2(x,z)$ and
$f_3(x,y) = x + y = l_3(x) + r_3(y)$ with $l_3(x) = x$. 
Perturbing then $m(x,y)$
according to Theorem~\ref{theorem:DS_quadratic}(i), 
$l_2(x,z)$ and $l_3(x)$ according to Lemma~\ref{lemma:DS_linear_2},
and $r_1(y)$ and $r_3(y)$ 
via Theorem~\ref{theorem:universal}, 
one obtains
\begin{align}
\frac{\mathrm{d} x}{\mathrm{d} t} & = - y + x^2 - \frac{\mu}{a} x y, 
\nonumber \\
\frac{\mathrm{d} y}{\mathrm{d} t} & = \frac{27}{10} x + z - \varepsilon y, 
\nonumber \\
\frac{\mathrm{d} z}{\mathrm{d} t} & = x + y - \varepsilon z
+ \frac{\mu}{c} y z.
\label{eq:1_perturbed}
\end{align}
Under the translation
$(\bar{x},\bar{y},\bar{z}) = (x,y,z) + (a,b,c)/\mu$, 
one then obtains
\begin{align}
\frac{\mathrm{d} \bar{x}}{\mathrm{d} t} 
& = \frac{a^2}{\mu^2} + \left(-\frac{2 a}{\mu} + \frac{b}{a} \right) \bar{x} + \bar{x}^2 - \frac{\mu}{a} \bar{x} \bar{y}, \nonumber \\
\frac{\mathrm{d} \bar{y}}{\mathrm{d} t} 
& = \frac{1}{\mu} \left(- \frac{27}{10} a - c + \varepsilon b\right)
+ \frac{27}{10} \bar{x} - \varepsilon \bar{y}  + \bar{z}, \nonumber \\
\frac{\mathrm{d} \bar{z}}{\mathrm{d} t} 
& = \frac{1}{\mu} (-a + \varepsilon c)
+ \bar{x} - \left(\frac{b}{c} + \varepsilon \right) \bar{z}
+ \frac{\mu}{c} \bar{y} \bar{z}.
\label{eq:1_perturbed_translated}
\end{align}
Sufficient to make this system chemical is by choosing
$b = a/\varepsilon^2 + 27a/(10 \varepsilon)$, 
$c = a/\varepsilon$ and e.g. $a = 1$.
Then, by rescaling $\bar{x} \to 
|\frac{1}{\varepsilon^2} 
+ \frac{27}{10 \varepsilon} - \frac{2}{\mu}| \bar{x}$
and $\bar{y} \to 27 \bar{y}/10$, the monomials
$\bar{x}$ are multiplied, up to sign, by the same
coefficients, thus allowing for fusion of the underlying 
three reactions (see Appendix~\ref{app:CRN}).
Removing the bars, DS~(\ref{eq:1_perturbed_translated}) 
then becomes the CDS~(\ref{eq:CDS_1}), and the 
theorem then follows from the assumption that 
the chaotic attractor of~(\ref{eq:DS_1})
and its one-wingedness are robust, 
and from the fact that one-wingedness is 
invariant under affine transformations.
\end{proof}

\begin{example}
Let us choose $\varepsilon = \mu = 10^{-2}$
in~$(\ref{eq:parameters_1})$.
The numerically observed one-wing chaotic attractor 
of \emph{CDS}~$(\ref{eq:CDS_1})$ is then shown  
in the $(x,y)$-space in \emph{Figure~\ref{fig:CDS_1}(d)},
while in \emph{Figure~\ref{fig:CDS_1}(e)} in the $(t,x)$-space.
The associated \emph{LCE}s are shown in~\emph{Figure~\ref{fig:CDS_1}(f)},
and are relatively close to those of~$(\ref{eq:DS_1})$.
See also \emph{Figure~\ref{fig:CDS_0}(a)} for 
the one-wing chaotic attractor in the $(x,y,z)$-space.
\end{example}

\subsection{Chaotic DSs with two quadratic monomials}
\label{sec:two_quadratics}
For quadratic DSs with only two quadratic terms, 
Theorem~\ref{theorem:DS_quadratic} implies the following result.

\begin{theorem} 
\label{theorem:two_quadratic}
Assume that $N$-dimensional quadratic \emph{DS},
with only two quadratic monomials, 
has a robust bounded invariant set in 
$\mathbb{R}^N$ with robust properties 
$\mathcal{P}$ that are invariant under permutations, reflections 
and translations of the dependent variables. 
Then, there exists an $N$-dimensional cubic \emph{CDS},
with at most one cubic monomial,
with a bounded invariant set in $\mathbb{R}_{>}^N$ 
with properties $\mathcal{P}$. 
\end{theorem}

\begin{proof}
The statement of the theorem follows by 
analogous arguments as for Theorem~\ref{theorem:one_quadratic}.
\end{proof}

Using the Lorenz system~\cite{Lorenz}, one
can then deduce the following result,
which can be seen as an improvement of~\cite{QCM}[Theorem 5.2]. 
\begin{corollary}
\label{corollary:Lorenz}
There exists a three-dimensional cubic \emph{CDS},
with only one cubic monomial, 
that has a chaotic attractor in $\mathbb{R}_{>}^3$.
\end{corollary}

\begin{proof}
The corollary follows from Theorem~\ref{theorem:two_quadratic}
and the fact that the three-dimensional quadratic 
$(7,2)$ Lorenz system has a robust chaotic attractor~\cite{Tucker}.
\end{proof}

\subsubsection{Two-wing chaos}
The three-dimensional quadratic $(7,2)$ Lorenz
and the $(5,2)$ Sprott systems B--C~\cite{Sprott}[Table 1]
display two-wing chaotic attractors, and can
by Theorem~\ref{theorem:DS_quadratic} all be mapped
to three-dimensional cubic CDSs with only one cubic.
In what follows, we perform this construction on
the Sprott system C; in particular, consider the DS
\begin{align}
\frac{\mathrm{d} x}{\mathrm{d} t} & = -1 + y^2, \nonumber \\
\frac{\mathrm{d} y}{\mathrm{d} t} & = - x z, \nonumber \\
\frac{\mathrm{d} z}{\mathrm{d} t} & = y - z,
\label{eq:DS_2}
\end{align}
obtained via permutation $y \to z$, then $x \to y$
and then reflection $x \to - x$ 
in the Sprott system C.
See Figure~\ref{fig:CDS_2}(a)--(c)
for the underlying chaotic attractor and the LCEs.

\begin{figure}[!htbp]
\vskip -2.0cm
\leftline{\hskip 
0.1cm (a) DS~(\ref{eq:DS_2}) \hskip  
4.1cm (b) \hskip  
5.3cm (c)}
\centerline{
\hskip 0.0cm
\includegraphics[width=0.35\columnwidth]{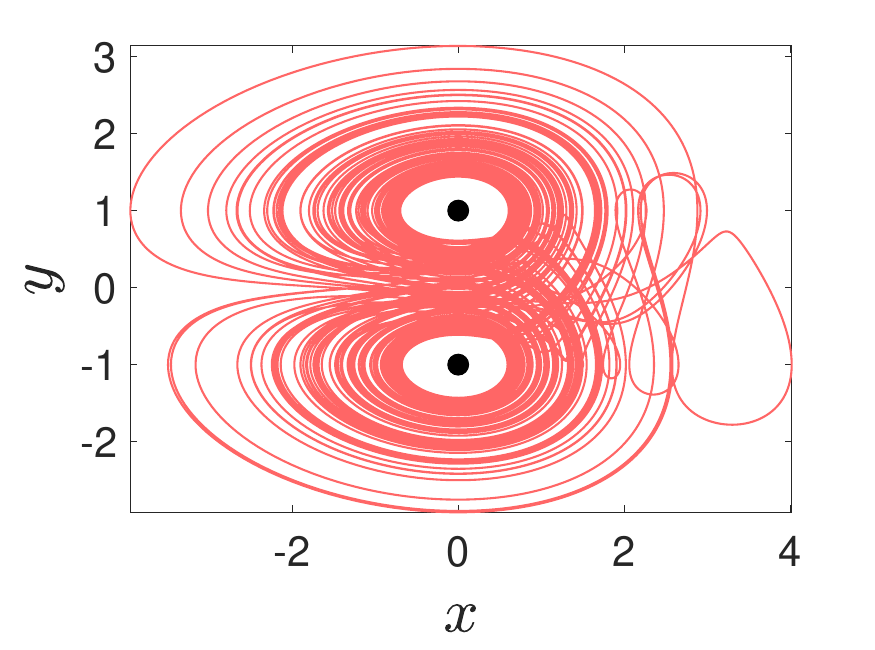}
\hskip 0.2cm
\includegraphics[width=0.35\columnwidth]{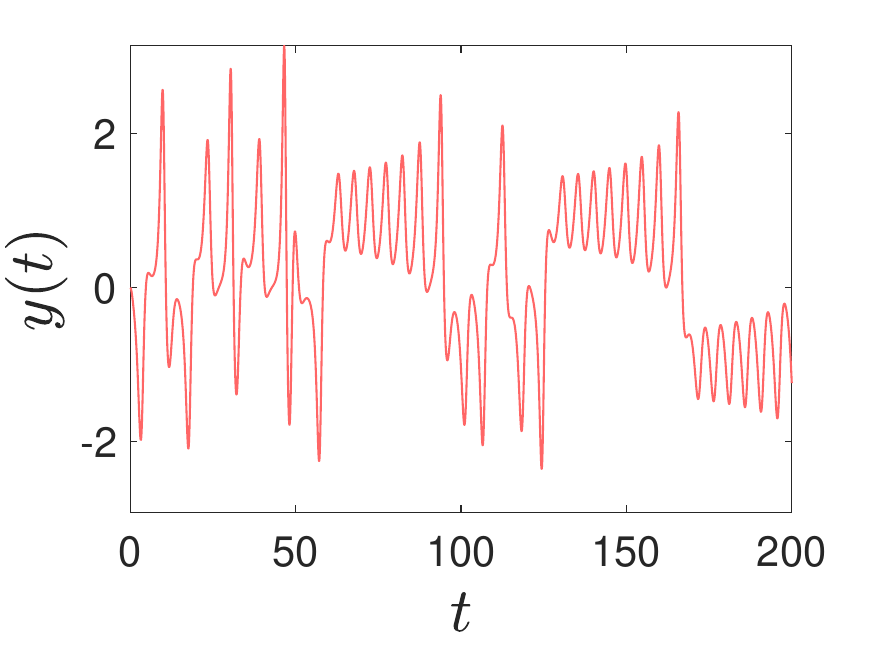}
\hskip 0.2cm
\includegraphics[width=0.35\columnwidth]{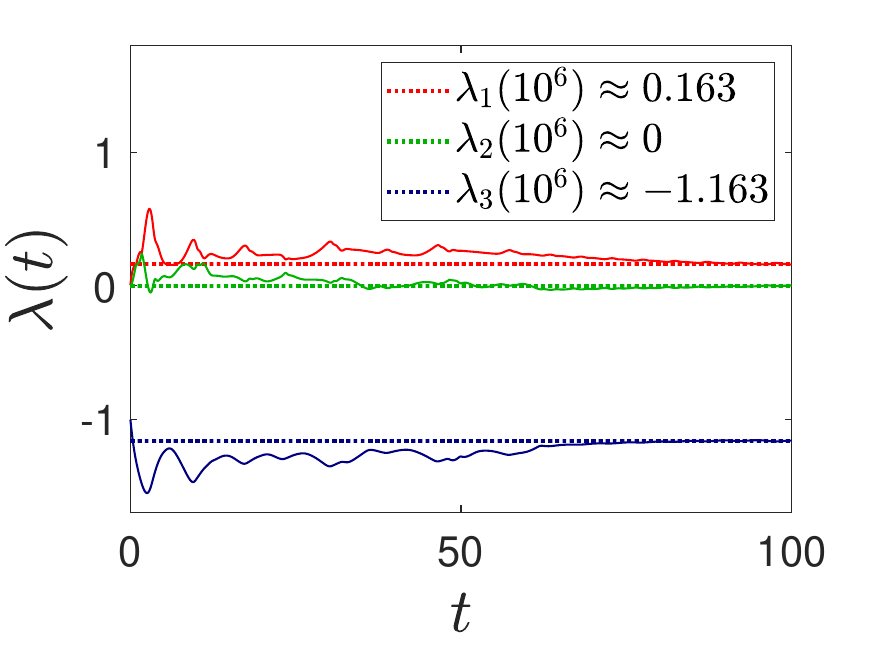}
}
\leftline{\hskip 
0.1cm (d) CDS~(\ref{eq:CDS_2}) \hskip  
3.9cm (e) \hskip  
5.3cm (f)}
\vskip 0.2cm
\centerline{
\hskip 0.0cm
\includegraphics[width=0.35\columnwidth]{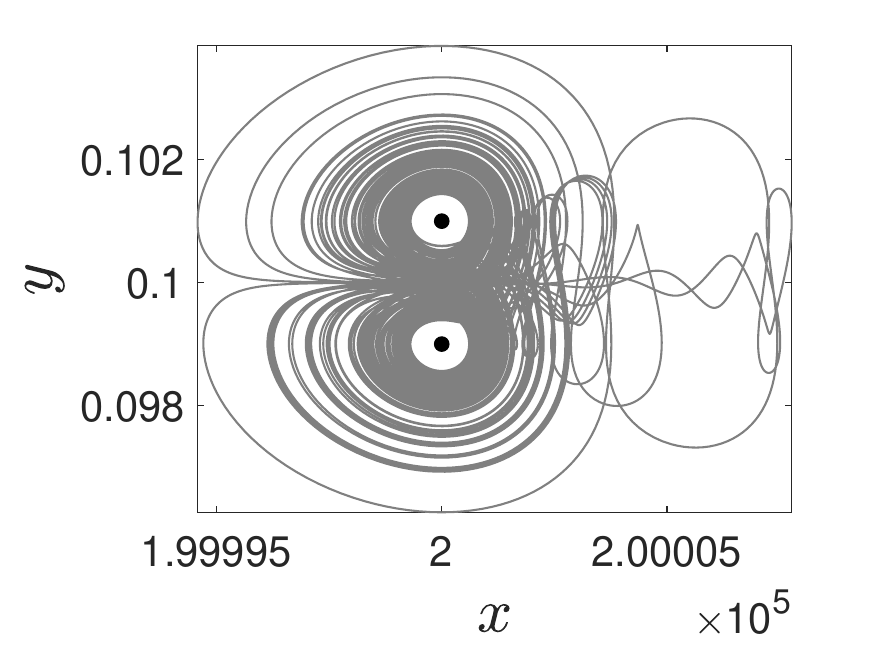}
\hskip 0.2cm
\includegraphics[width=0.35\columnwidth]{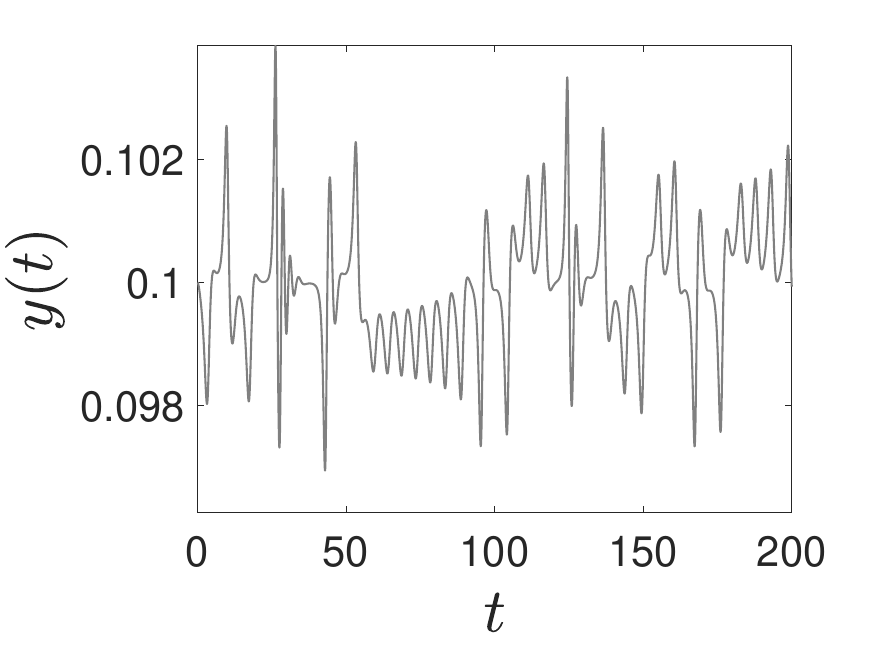}
\hskip 0.2cm
\includegraphics[width=0.35\columnwidth]{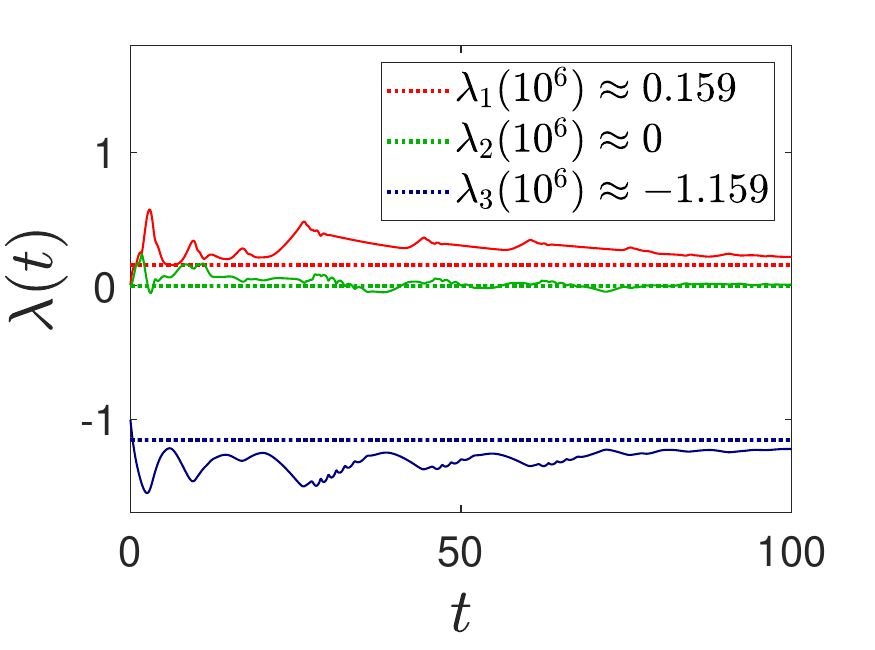}
}
\vskip -0.2cm
\caption{\it{\emph{Cubic $(11,5,1)$ CDS with two-wing chaos.} 
Panels \emph{(a)}--\emph{(b)} display respectively the 
$(x,y)$- and $(t,y)$-space for \emph{DS}~$(\ref{eq:DS_2})$
with initial condition $(x_0,y_0,z_0) = (0,0,-1)$;
the two equilibria are shown as black dots in panel \emph{(a)}.
Panel \emph{(c)} shows
the three finite-time \emph{LCE}s, 
along with the values
at $t = 10^6$ shown as the dashed lines.
Panels \emph{(d)}--\emph{(e)} show analogous plots 
for \emph{CDS}~$(\ref{eq:CDS_2})$
with $(\varepsilon,\mu) = (10^{-3},10^{-2})$
and  the translated and rescaled initial condition 
$x_0 = 2 (\varepsilon \mu)^{-1}$, 
$y_0 = \varepsilon \mu^{-1}$, 
$z_0 = 2 \mu^{-1} (-1 + \mu^{-1})$. 
Panel \emph{(f)} shows the \emph{LCE}s 
for the perturbed \emph{DS}~$(\ref{eq:2_perturbed})$
with $b = 1$, which are in the long-run identical
to the long-time \emph{LCE}s of~$(\ref{eq:CDS_2})$.}} 
\label{fig:CDS_2}
\end{figure}

\begin{theorem}$($\textbf{\emph{Two-wing chaos}}$)$ 
\label{theorem:two_wings}
Consider three-dimensional cubic $(11,5,1)$ \emph{CDS}
\begin{align}
\frac{\mathrm{d} x}{\mathrm{d} t} 
& = \alpha_1 - \alpha_2 x
+ \alpha_3 x^2 + \alpha_4 y^2 - \alpha_5 x y, 
\nonumber \\
\frac{\mathrm{d} y}{\mathrm{d} t} 
& = - \alpha_6 y
+ \alpha_5 x y
+ \alpha_7 y z
- \alpha_8 x y z, \nonumber \\
\frac{\mathrm{d} z}{\mathrm{d} t} 
& = \alpha_{6} y - \alpha_{9} z,
\label{eq:CDS_2}
\end{align}
with $(9,4,1)$ \emph{CRN}
\begin{align}
\varnothing & \xrightarrow[]{\alpha_1} X, \; \; \; 
X \xrightarrow[]{\alpha_2} \varnothing, \; \; \; 
2 X \xrightarrow[]{\alpha_3} 3 X, \; \; \; 
2 Y \xrightarrow[]{\alpha_4} X + 2 Y,  \; \; \; 
X + Y \xrightarrow[]{\alpha_5} 2 Y, \nonumber \\
Y & \xrightarrow[]{\alpha_6} Z , \; \; \; 
Y + Z \xrightarrow[]{\alpha_7} 2 Y + Z,  \; \; \; 
X + Y + Z \xrightarrow[]{\alpha_8} X + Z, \; \; \;
Z \xrightarrow[]{\alpha_9} \varnothing,
\label{eq:CRN_2}
\end{align}
and parameters
\begin{align}
\alpha_1 & = \frac{4}{\varepsilon \mu^2} - \frac{1}{\mu^2} - 1, \; \; 
\alpha_2 = \frac{4}{\mu} - \frac{\varepsilon}{\mu}, \; \;
\alpha_3 = \varepsilon, \; \;  
\alpha_4 = \frac{1}{\varepsilon^2}, \nonumber \\
\alpha_5 & = 1, \; \; 
\alpha_6 = \frac{2}{\varepsilon \mu}, \; \; 
\alpha_7 = \frac{\mu}{\varepsilon}, \; \; 
\alpha_8 = \frac{\mu^2}{2}, \; \; 
\alpha_9 = 1.
\label{eq:parameters_2}
\end{align}
Assume that \emph{DS}~$(\ref{eq:DS_2})$
has a robust two-wing chaotic attractor in $\mathbb{R}^3$. 
Then, \emph{CDS}~$(\ref{eq:CDS_2})$ has a two-wing
chaotic attractor in $\mathbb{R}_{>}^3$
for every sufficiently small $\varepsilon > 0$ and $\mu > 0$. 
\end{theorem}

\begin{proof}
Let us split DS~(\ref{eq:DS_2}) as follows:
$f_1(y) = -1 + y^2 = q_1(y)$, 
$f_2(x,z) = - x z = r_2(x,z)$ and
$f_3(y,z) = y - z = l_3(y,z)$. 
Perturbing then $q_1(y)$
according to Theorem~\ref{theorem:DS_quadratic}(iii), 
$l_3(y,z)$ according to Lemma~\ref{lemma:DS_linear_2},
and $r_2(x,z)$ via Theorem~\ref{theorem:universal}, 
one obtains
\begin{align}
\frac{\mathrm{d} x}{\mathrm{d} t} & = -1 + y^2
+ \varepsilon (x^2 - x y), \nonumber \\
\frac{\mathrm{d} y}{\mathrm{d} t} & = - x z
- \frac{\mu}{b} x y z, \nonumber \\
\frac{\mathrm{d} z}{\mathrm{d} t} & = y - z.
\label{eq:2_perturbed}
\end{align}
Under the translation
$(\bar{x},\bar{y},\bar{z}) = (x,y,z) + (a,b,c)/\mu$, 
one then obtains
\begin{align}
\frac{\mathrm{d} \bar{x}}{\mathrm{d} t} 
& = \left(\frac{1}{\mu^2} \left(b^2 + \varepsilon a^2 - \varepsilon a b \right) - 1 \right)
+ \frac{\varepsilon}{\mu} (b - 2 a) \bar{x}
+ \frac{1}{\mu} \left(- 2 b + \varepsilon a\right) \bar{y}
+ \varepsilon \bar{x}^2 + \bar{y}^2 - \varepsilon \bar{x} \bar{y}, 
\nonumber \\
\frac{\mathrm{d} \bar{y}}{\mathrm{d} t} 
& = -\frac{1}{\mu} \frac{a c}{b} \bar{y}
+ \frac{c}{b} \bar{x} \bar{y}
+ \frac{a}{b} \bar{y} \bar{z}
- \frac{\mu}{b} \bar{x} \bar{y} \bar{z}, \nonumber \\
\frac{\mathrm{d} \bar{z}}{\mathrm{d} t} 
& = \frac{1}{\mu} (c - b)
+ \bar{y} - \bar{z}.
\label{eq:2_perturbed_translated}
\end{align}
To make this system chemical, it is sufficient
to choose $a = 2 b/\varepsilon$, $c = b$ and e.g. $b = 1$.
By then rescaling via $\bar{y} \to \bar{y}/\varepsilon$
and $\bar{z} \to \mu \bar{z}/2$, 
the monomials $\bar{x} \bar{y}$, as well as $\bar{y}$,
are multiplied, up to sign, by the same
coefficients, thus allowing for fusion of the underlying 
reactions (see Appendix~\ref{app:CRN}). 
Removing the bars, DS~(\ref{eq:2_perturbed_translated}) 
then becomes the CDS~(\ref{eq:CDS_2}), and the
theorem then follows from the assumption that the chaotic attractor
of~(\ref{eq:DS_2}) and its two-wingedness are robust, and 
the fact that two-wingedness is invariant under affine transformations.
\end{proof}

\begin{example}
Let us choose $\varepsilon = 10^{-3}$ and $\mu = 10^{-2}$
in~$(\ref{eq:parameters_2})$.
The numerically observed two-wing chaotic attractor 
of \emph{CDS}~$(\ref{eq:CDS_2})$
is then shown in the $(x,y)$-space in \emph{Figure~\ref{fig:CDS_2}(d)};
the $(t,y)$-space is shown in \emph{Figure~\ref{fig:CDS_2}(e)}.
The corresponding \emph{LCE}s are displayed 
in~\emph{Figure~\ref{fig:CDS_2}(f)},
and are relatively close to those of \emph{DS}~$(\ref{eq:DS_2})$.
See also \emph{Figure~\ref{fig:CDS_0}(b)} for 
the two-wing chaotic attractor in the $(x,y,z)$-space.
\end{example}

\subsubsection{Hidden chaos}
\label{sec:hidden}
The chaotic attractors considered in the previous sections
can be detected via equilibria: by choosing an initial
condition sufficiently close to some unstable equilibrium, 
the corresponding trajectory converges to the chaotic attractor
in the long-run. In this section, we consider chaotic attractors
that are hidden from equilibria~\cite{Hidden}.

\begin{definition} $($\textbf{Hidden chaos}$)$ 
\label{def:chaos_hidden}
Consider \emph{DS}~$(\ref{eq:DS})$ with a chaotic attractor 
$\mathbb{V}$. Assume that the following statement is false:
for some equilibrium $\mathbf{x}^*$ of~$(\ref{eq:DS})$
and for every $\varepsilon > 0$ there exists
an initial condition $\mathbf{x}_0$ with 
$\|\mathbf{x}_0 - \mathbf{x}^*\| < \varepsilon$
such that the solution $\mathbf{x}(t;\mathbf{x}_0)$
of~$(\ref{eq:DS})$ converges to $\mathbb{V}$
as $t \to \infty$. 
Then, $\mathbb{V}$ is said to be a \emph{hidden} 
chaotic attractor of~$(\ref{eq:DS})$.
\end{definition}

It follows from Definition~\ref{def:chaos_hidden}
that sufficient for a chaotic attractor to be hidden
is that the corresponding DS has  no equilibria, 
or only stable equilibria. In this context, 
$23$ three-dimensional quadratic DSs 
with chaos and a unique and stable equilibrium
have been reported in~\cite{Sprott2};
these DSs contain $7$ or $8$ monomials in total, 
of which $2$ or $3$ are quadratic. 
Let us consider the quadratic DS
\begin{align}
\frac{\mathrm{d} x}{\mathrm{d} t} & = \frac{57}{100} 
- \frac{31}{10} z - \frac{1}{5} x y - \frac{3}{10} x z, \nonumber \\
\frac{\mathrm{d} y}{\mathrm{d} t} & = - y - z, \nonumber \\
\frac{\mathrm{d} z}{\mathrm{d} t} & = x,
\label{eq:DS_3}
\end{align}
obtained via permutation $x \to z$ and reflection $y \to -y$
in $\textrm{SE}_{17}$ from~\cite{Sprott2}[Table 1].
The hidden chaotic attractor, and 
the unique and stable equilibrium, 
are displayed in the $(x,y)$-space in Figure~\ref{fig:CDS_3}(a);
the trajectory in the $(t,x)$-space is shown in 
Figure~\ref{fig:CDS_3}(b), 
with the LCEs in Figure~\ref{fig:CDS_3}(c).

\begin{figure}[!htbp]
\vskip -2.0cm
\leftline{\hskip 
0.1cm (a) DS~(\ref{eq:DS_3}) \hskip  
4.1cm (b) \hskip  
5.3cm (c)}
\centerline{
\hskip 0.0cm
\includegraphics[width=0.35\columnwidth]{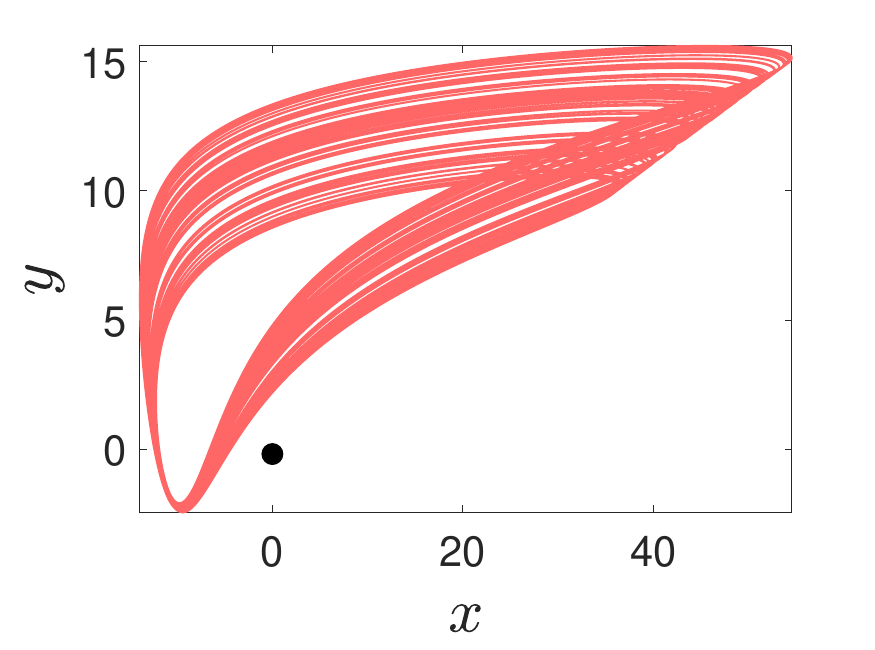}
\hskip 0.2cm
\includegraphics[width=0.35\columnwidth]{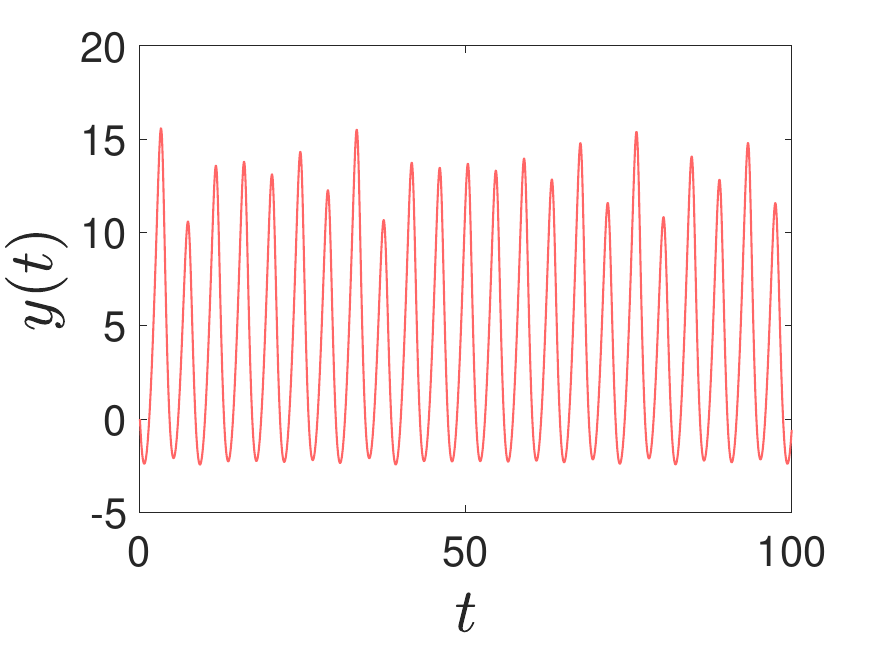}
\hskip 0.2cm
\includegraphics[width=0.35\columnwidth]{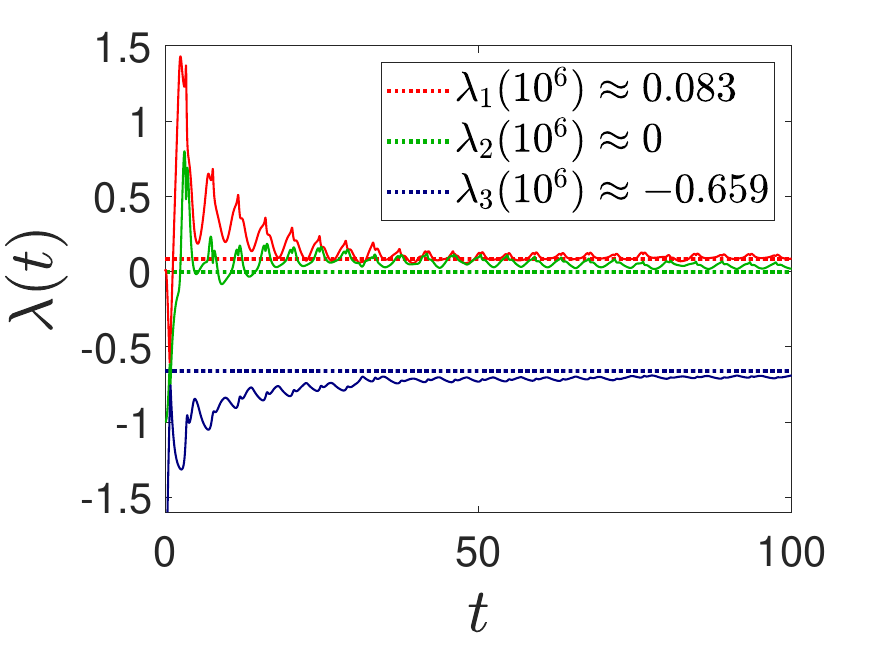}
}
\leftline{\hskip 
0.1cm (d) CDS~(\ref{eq:CDS_3}) \hskip  
3.9cm (e) \hskip  
5.3cm (f)}
\vskip 0.2cm
\centerline{
\hskip 0.0cm
\includegraphics[width=0.35\columnwidth]{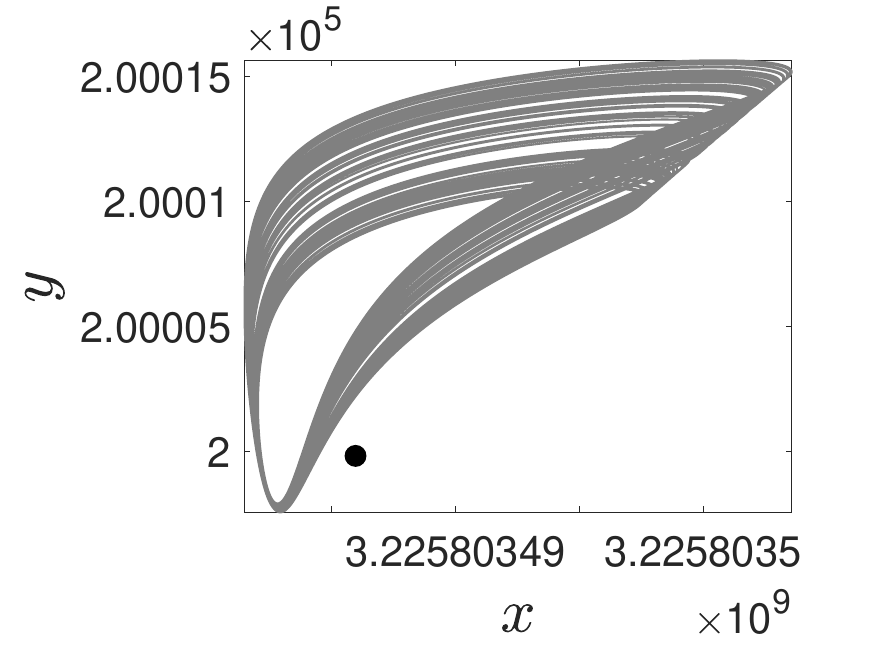}
\hskip 0.2cm
\includegraphics[width=0.35\columnwidth]{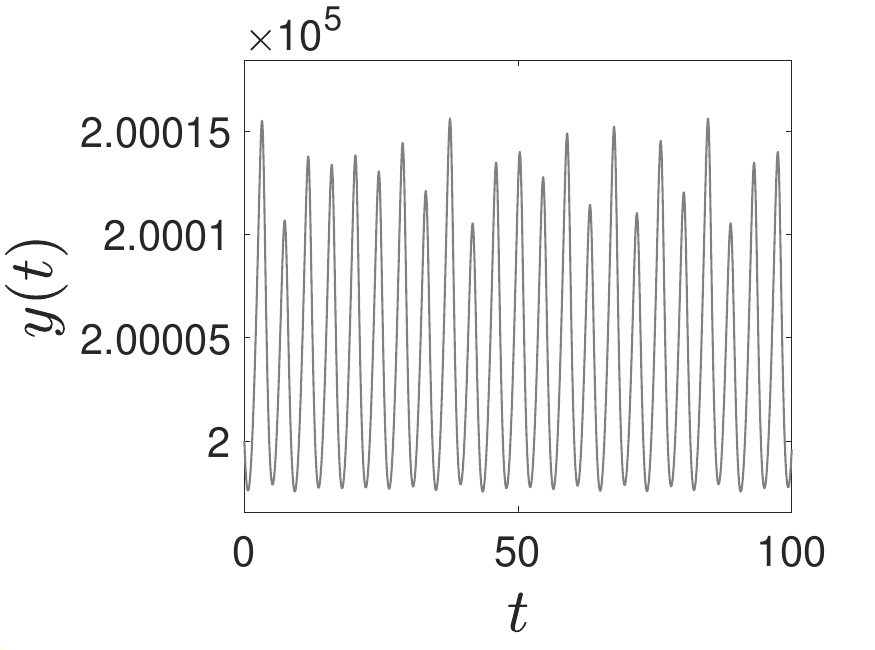}
\hskip 0.2cm
\includegraphics[width=0.35\columnwidth]{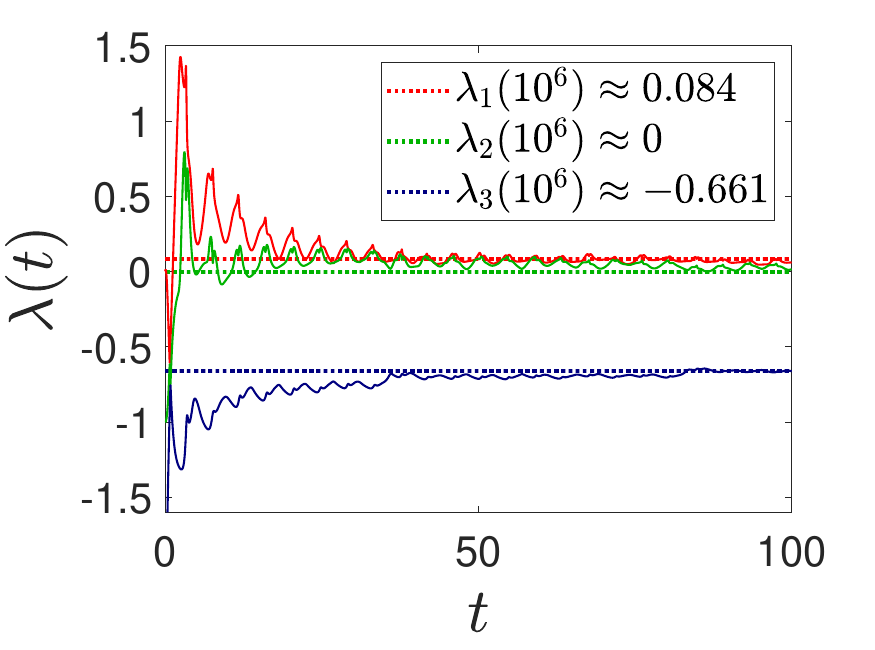}
}
\vskip -0.2cm
\caption{\it{\emph{Quadratic $(11,5)$ CDS with hidden chaos.} 
Panels \emph{(a)}--\emph{(b)} respectively display the $(x,y)$-
and $(t,y)$-space for \emph{DS}~$(\ref{eq:DS_3})$,
with initial condition $(x_0,y_0,z_0) = (-5,0,7.5)$;
the stable equilibrium is shown as a
black dot in \emph{(a)}.
Panel \emph{(c)} shows
the three finite-time \emph{LCE}s, 
along with the values
at $t = 10^6$ shown as the dashed lines.
Panels \emph{(d)}--\emph{(e)} show analogous plots 
for \emph{CDS}~$(\ref{eq:CDS_3})$
with $(\varepsilon,\mu) = (10^{-5}, 10^{-5})$ and 
the translated and rescaled initial condition 
$x_0 = (3.1 + 
0.285 \mu)^{-1} (-5 + (\varepsilon \mu)^{-1})$, 
$y_0 = 2 \mu^{-1}$, 
$z_0 = 0.6 \mu^{-1} (3.1 + 0.285 \mu)^{-1} (7.5 + 2 \mu^{-1})$.
Panel \emph{(f)} shows the \emph{LCE}s 
for the perturbed \emph{DS}~$(\ref{eq:3_perturbed})$
with $b = c = 2$, which are in the long-run identical
to the long-time \emph{LCE}s of~$(\ref{eq:CDS_3})$.}} 
\label{fig:CDS_3}
\end{figure}

Theorem~\ref{theorem:DS_quadratic} guarantees that~(\ref{eq:DS_3})
can be mapped to a one-cubic-term CDS, but not necessarily 
to a quadratic CDS. Furthermore, (smooth) perturbations
of the vector field in general introduce
additional equilibria, i.e. having a unique
equilibrium is not a robust property of DSs;
consequently, CDSs obtained from~(\ref{eq:DS_3}) 
via QCMs in general have multiple equilibria. 
We now design a QCM that overcomes these issues:
both the degree and equilibrium uniqueness 
of~(\ref{eq:DS_3}) are preserved. 

\begin{theorem}$($\textbf{\emph{Hidden chaos}}$)$ 
\label{theorem:hidden}
Consider three-dimensional quadratic $(11,5)$ \emph{CDS}
\begin{align}
\frac{\mathrm{d} x}{\mathrm{d} t} 
& = \alpha_1 - \alpha_2 x + \alpha_3 y + \alpha_4 z
+ \alpha_5 x^2 - \alpha_6 x y - \alpha_7 x z, \nonumber \\
\frac{\mathrm{d} y}{\mathrm{d} t} 
& = \alpha_1 - \alpha_8 y z, 
\nonumber \\
\frac{\mathrm{d} z}{\mathrm{d} t} 
& = -\alpha_9 z + \alpha_7 x z,
\label{eq:CDS_3}
\end{align}
with $(9,4)$ \emph{CRN}
\begin{align}
\varnothing & \xrightarrow[]{\alpha_1} X + Y, \; \; \; 
X \xrightarrow[]{\alpha_2} \varnothing,  \; \; \; 
Y \xrightarrow[]{\alpha_3} Y + X, \; \; \;
Z \xrightarrow[]{\alpha_4} Z + X, \nonumber \\
2 X & \xrightarrow[]{\alpha_5} 3 X, \; \; \; 
X + Y \xrightarrow[]{\alpha_6} Y,   \; \; \; 
X + Z \xrightarrow[]{\alpha_7} 2 Z, \; \; \; 
Y + Z \xrightarrow[]{\alpha_8} Z,  \; \; \; 
Z \xrightarrow[]{\alpha_9} \varnothing,
\label{eq:CRN_3}
\end{align}
and parameters
\begin{align}
\alpha_1 & = \frac{2}{\mu}, \; \; 
\alpha_2 = \frac{1}{\mu}, \; \; 
\alpha_3 = \frac{40}{620 \varepsilon \mu + 57 \varepsilon \mu^2}, \; \; \alpha_4 = \frac{3 - 31 \varepsilon \mu}{6 \varepsilon}, \; \; 
\alpha_5 = \frac{\varepsilon (620 + 57 \mu)}{200}, \nonumber \\
\alpha_6 & = \frac{1}{5}, \; \; 
\alpha_7 = \frac{\mu (620  + 57 \mu)}{400}, \; \; 
\alpha_8 = \frac{\mu^2 (620 + 57 \mu)}{240}, \; \; 
\alpha_9 = \frac{1}{2 \varepsilon}.
\label{eq:parameters_3}
\end{align}
Assume that \emph{DS}~$(\ref{eq:DS_3})$
has a robust chaotic attractor in $\mathbb{R}^3$. 
Then, \emph{CDS}~$(\ref{eq:CDS_3})$ has a unique and stable equilibrium,
and a hidden chaotic attractor, in $\mathbb{R}_{>}^3$
for every sufficiently small $\varepsilon > 0$ and $\mu > 0$. 
\end{theorem}

\begin{proof}
Let us split DS~(\ref{eq:DS_3}) as follows:
$f_1(x,y,z) = 57/100
- 31 z/10 - x y/5 - 3 x z/10 
= q_1(x,y,z)$, 
$f_2(y,z) = - y - z = l_2(y) + r_2(z)$ with $l_2(y) = - y$, and
$f_3(x) = x = r_3(x)$. 
Perturbing then $q_1(x,y,z)$ with $\varepsilon x^2$,
to satisfy the conditions~(\ref{eq:quadratic_constraint}) from Lemma~\ref{lemma:DS_quadratic},
$l_2(y)$ according to Lemma~\ref{lemma:DS_linear_2},
and $r_2(z)$ and $r_3(x)$ via Theorem~\ref{theorem:universal}, 
one obtains
\begin{align}
\frac{\mathrm{d} x}{\mathrm{d} t} & =
\frac{57}{100} - \frac{31}{10} z - \frac{1}{5} x y
- \frac{3}{10} x z + \varepsilon x^2, \nonumber \\
\frac{\mathrm{d} y}{\mathrm{d} t} & = - y - z
- \frac{\mu}{b} y z, \nonumber \\
\frac{\mathrm{d} z}{\mathrm{d} t} & = x
+ \frac{\mu}{c} x z.
\label{eq:3_perturbed}
\end{align}
One equilibrium of DS~(\ref{eq:3_perturbed}) is given by
\begin{align}
(x^*,y^*,z^*) & = \left(0, 
\frac{-57}{310 + \frac{57}{b} \mu}, \frac{57}{310}\right),
\label{eq:3_perturbed_eq}
\end{align}
which, for every sufficiently small $\varepsilon > 0$ and $\mu > 0$, 
being a perturbation of the hyperbolic and stable 
equilibrium of~(\ref{eq:DS_3}),
is itself hyperbolic and stable. Every other equilibrium
$(x^*,y^*,z^*)$ of~(\ref{eq:3_perturbed}) satisfies
$0 = (c/b - 1) y^* + c/\mu$; to this end, 
we choose $c = b$, which ensures 
that the equilibrium~(\ref{eq:3_perturbed_eq}) is unique. 
Translating then via 
$(\bar{x},\bar{y},\bar{z}) = (x,y,z) + (a,b,b)/\mu$, 
one obtains
\begin{align}
\frac{\mathrm{d} \bar{x}}{\mathrm{d} t} 
& = \left[\frac{a}{\mu^2} \left(\varepsilon a
- \frac{b}{2} \right) + \frac{31 b}{10 \mu} + \frac{57}{100} \right]
+ \frac{1}{10 \mu} \left(5 b - 20 \varepsilon a \right) \bar{x}
+ \frac{a}{5 \mu} \bar{y}  
+ \left(-\frac{31}{10} + \frac{3 a}{10 \mu} \right) \bar{z}\nonumber \\
& + \varepsilon \bar{x}^2
- \frac{1}{5} \bar{x} \bar{y}
- \frac{3}{10} \bar{x} \bar{z}, \nonumber \\
\frac{\mathrm{d} \bar{y}}{\mathrm{d} t} 
& = \frac{b}{\mu} - \frac{\mu}{b} \bar{y} \bar{z}, \nonumber \\
\frac{\mathrm{d} \bar{z}}{\mathrm{d} t} 
& = - \frac{a}{b} z + \frac{\mu}{b} \bar{x} \bar{z}.
\label{eq:3_perturbed_translated}
\end{align}
Sufficient to make this system chemical is with 
$b = 2 \varepsilon a$; we then e.g. fix  
$a = 1/\varepsilon$ to make $b = \mathcal{O}(1)$.
Then, rescaling via $\bar{x} \to (31/10 + 
57 \mu /200) \bar{x}$
and $\bar{z} \to (5/3) \mu 
(31/10 + 57 \mu /200)\bar{z}$, 
the monomials $\bar{x} \bar{z}$, as well as $1$,
are multiplied, up to sign, by the same
coefficients, thus allowing for fusion of the underlying 
reactions (see Appendix~\ref{app:CRN}). 
Removing the bars, DS~(\ref{eq:3_perturbed_translated}) 
then becomes the CDS~(\ref{eq:CDS_3}), and the
theorem then follows from the assumption that the chaotic attractor
of~(\ref{eq:DS_3}) is robust.
\end{proof}

\begin{example}
Let us choose $\varepsilon = \mu = 10^{-5}$ 
in~$(\ref{eq:parameters_3})$, under which the 
unique equilibrium of~$(\ref{eq:CDS_3})$ is stable.
The numerically observed hidden chaotic attractor,
together with the equilibrium, are shown 
in the $(x,y)$-space in \emph{Figure~\ref{fig:CDS_3}(d)};
the corresponding $(t,y)$-space is shown in 
\emph{Figure~\ref{fig:CDS_3}(e)}.
The associated \emph{LCE}s are displayed in 
\emph{Figure~\ref{fig:CDS_3}(f)}, and are close
to those of the \emph{DS}~$(\ref{eq:DS_3})$.
See also \emph{Figure~\ref{fig:CDS_0}(c)} for 
the hidden chaotic attractor and the equilibrium 
in the $(x,y,z)$-space.
\end{example}

\section{Discussion}
\label{sec:discussion}
In this paper, we have developed some fundamental theory 
about mapping polynomial dynamical systems 
(DSs) into dynamically similar chemical dynamical systems (CDSs)
of the same dimension and with a reduced 
number of non-linear terms. 
We have then applied this theory to investigate 
the capacity of chemical and biological systems
to display chaotic behaviors with various properties. 

In particular, in Section~\ref{sec:DS_to_CDS}, 
we have developed some theory about 
the so-called \emph{quasi-chemical maps} (QCMs).
These maps, put forward in~\cite{QCM}, 
can transform polynomial DSs into CDSs
purely via smooth perturbations of the vector field 
and translations of the dependent variables,
thereby preserving both the dimension
and robust structures.
In Section~\ref{sec:linear}, we have constructed
some QCMs that can map linear DSs into 
quadratic CDSs with a lower number of quadratics.
Using one such map, we have proved 
in Theorem~\ref{theorem:DS_linear} that every 
$N$-dimensional linear DS, with $M_{1}^{-}$ 
non-chemical linear monomials, can be mapped
to an $N$-dimensional quadratic CDS
with at most $(M_{1}^{-} + N)$ quadratics.
Similarly, in Section~\ref{sec:quadratic}, 
we have constructed QCMs that can map quadratic
DSs into cubic CDSs with a lower number of cubics.
One such map has been used to prove in 
Theorem~\ref{theorem:DS_quadratic} that every 
quadratic DS, with $M_{2}$ 
quadratic monomials, can be mapped
to a cubic CDS with at most $(M_{2} - 1)$ cubics.
To reduce the number of both quadratics and cubics, 
these different QCMs can be combined into a single composite one
by suitably splitting DSs, as outlined in Lemma~\ref{lemma:splitting}.
This approach has then been used in Section~\ref{sec:applications} 
to investigate chaos in CDSs.

In particular, in Section~\ref{sec:one_quadratic}, we have 
constructively proved
in Theorem~\ref{theorem:one_quadratic} that every 
quadratic DS, with only one quadratic monomial
and robust chaos, can be mapped to a quadratic CDS
of the same dimension and with the robust chaos preserved.
A number of such DSs have been put 
forward, such as the R\"ossler~\cite{Rossler}
and Sprott systems F--S~\cite{Sprott}.
Assuming existence of robust chaos in these DSs, Theorem~\ref{theorem:one_quadratic} 
then shows that three-dimensional quadratic CDSs can display a range
of chaotic behaviors. As a concrete example, 
we have constructed in Theorem~\ref{theorem:Rossler}
the chemical R\"ossler system - a quadratic $(11,5)$ CDS,
with $11$ monomials of which exactly $5$ are quadratic,
with $(9,4)$ chemical reaction network (CRN),
with $9$ reactions of which exactly $4$ are quadratic,
displaying the R\"ossler attractor.
This CDS is simpler, when it comes to the number of quadratics, 
than the minimal $(9,6)$ Willamowski–R\"ossler 
system~\cite{RosslerW,MinRosslerW}  with $(7,4)$ CRN.
We have then constructed in Theorem~\ref{theorem:one_wing}
a quadratic $(10,3)$ CDS with $(8,3)$ CRN and a chaotic attractor
with one wing, shown in Figure~\ref{fig:CDS_0}(a).
To the best of the author's knowledge, this 
is currently the three-dimensional quadratic CDS with
observed chaos that has the lowest number of quadratic
monomials and reactions.

In Section~\ref{sec:two_quadratics}, 
we have constructively proved in Theorem~\ref{theorem:two_quadratic} 
that every quadratic DS, with only two quadratic monomials
and robust chaos, can be mapped to a cubic CDS,
with at most one cubic monomial,
of the same dimension and with the robust chaos preserved.
A number of such DSs have been reported, 
such as the Lorenz~\cite{Lorenz} and
Sprott systems A--E~\cite{Sprott}[Table 1], 
and many others~\cite{Sprott2,Hidden}.
Existence of a robust chaotic attractor
is proved for the Lorenz system~\cite{Tucker};
hence, one obtains Corollary~\ref{corollary:Lorenz},
guaranteeing existence of a three-dimensional 
cubic CDS with only one cubic and a chaotic attractor.
More broadly, assuming the other DSs also have robust chaos, 
Theorem~\ref{theorem:two_quadratic} shows
that three-dimensional cubic CDSs can realize 
a rich set of chaotic behaviors already with 
only one cubic monomial.
To showcase this, we have constructed a cubic $(11,5,1)$ CDS
with $(9,4,1)$ CRN in Theorem~\ref{theorem:two_wings}
that displays a chaotic attractor
with two wings; see Figure~\ref{fig:CDS_0}(b). Furthermore, in Theorem~\ref{theorem:hidden},
we have presented a quadratic $(11,5)$ CDS with $(9,4)$ CRN
displaying a hidden chaotic attractor, and a stable and unique
equilibrium, displayed in Figure~\ref{fig:CDS_0}(c). 
To the best of the author's
knowledge, these two CDSs have the lowest number
of highest-degree monomials and reactions reported to this day
in the context of the two-wing and hidden chaos;
in fact, (\ref{eq:CDS_3}) appears to be the first 
three-dimensional CDS with hidden chaos
and unique equilibrium reported.

We close this paper with three sets of 
open questions about three-dimensional CDSs with chaos.
Firstly, is there a quadratic CDS with a (one-wing) chaotic attractor
and less than $3$ quadratic monomials, or less than 
$3$ quadratic reactions?
Is there a \emph{quadratic} CDS with 
a two-wing chaotic attractor?
Is there a CDS with a chaotic attractor
and stable and unique equilibrium with less than 
$5$ quadratic monomials, or less than $4$ quadratic reactions?
Is there a cubic or quadratic CDS with 
a chaotic attractor and no equilibria?
What is the lowest number of linear monomials in all such CDSs?
Secondly, the QCMs, used to systematically construct the CDSs in this paper, 
rely on the perturbation parameters $\varepsilon, \mu > 0$, 
which have to take sufficiently small values in order for the 
target chaotic set to persist under the perturbations.
The ``less robust'' a chaotic set, the smaller one
requires the parameter values. Consequences of these
small values are that the polynomial coefficients 
in the constructed CDSs, as well as the regions in the state-space
where the chaotic sets are positioned,
can span many orders of magnitude, see e.g. (\ref{eq:CDS_Rossler_p})
and Figure~\ref{fig:CDS_0}.
Can one systematically construct CDSs of similar complexity
as in this paper, but without such scale separations?
Finally, chaotic attractors in the DSs considered in this paper
in general have a bounded region of attraction,
beyond which the solutions may grow unboundedly,
and the constructed CDSs can inherit such behaviors. 
Can one systematically construct 
globally bounded chaotic CDSs, without  
significantly increasing their structural complexity?

\appendix

\section{Appendix: Chemical reaction networks}
\label{app:CRN}
Every CDS induces a canonical set of chemical reactions~\cite{Janos}.
In what follows, given any $x \in \mathbb{R}$, we
define the sign function as
$\textrm{sign}(x) = -1$ if $x < 0$, 
$\textrm{sign}(x) = 0$ if $x = 0$, and 
$\textrm{sign}(x) = 1$ if $x > 0$.

\begin{definition} $($\textbf{Chemical reaction network}$)$ 
\label{def:CRN}
Assume that~$(\ref{eq:DS})$ is a \emph{CDS}. 
Let $m(\mathbf{x}) = \alpha x_1^{\nu_{1}} 
x_2^{\nu_{2}} \ldots x_N^{\nu_{N}}$
be a monomial from $f_i(\mathbf{x})$,
where $\alpha \in \mathbb{R}$
and $\nu_1,\nu_2,\ldots,\nu_N$ are non-negative integers.
Then, the monomial $m(\mathbf{x})
= \textrm{\emph{sign}}(\alpha) 
|\alpha| x_1^{\nu_{1}} x_2^{\nu_{2}} \ldots x_N^{\nu_{N}}$
induces the \emph{canonical chemical reaction}
\begin{align}
\sum_{k = 1}^N \nu_{k} X_k 
& \xrightarrow[]{|\alpha|} 
\left(\nu_{i} + \textrm{\emph{sign}}(\alpha) \right) X_i
+ \sum_{k = 1, k \ne i}^N \nu_{k} X_k, 
\label{eq:CR}
\end{align}
where $X_i$ denotes the chemical species whose concentration is $x_i$.
The set of all such chemical reactions,
induced by all the monomials in 
the vector field $\mathbf{f}(\mathbf{x})$,
is called the \emph{canonical chemical reaction network} (\emph{CRN}) 
induced by~$(\ref{eq:DS})$. 
\end{definition}
\noindent \textbf{Remark}. 
Terms of the form $0 X_i$ are denoted by $\varnothing$,
and interpreted as some neglected species.

For any given CDS, the induced canonical CRN 
from Definition~\ref{def:CRN} is unique.
However, a given CDS can also induce other, non-canonical, CRNs.

\begin{definition} $($\textbf{Fused reaction}$)$ \label{def:fused}
Consider $M$ canonical reactions with identical left-hand sides:
\begin{align}
\sum_{k = 1}^N \nu_{k} X_k & \xrightarrow[]{|\alpha_j|} 
\left(\nu_{k_j} + \textrm{\emph{sign}}(\alpha_j) \right) X_{k_j}
+ \sum_{k = 1, k \ne k_j}^N \nu_{k} X_k, 
\; \; \; j = 1,2, \ldots, M. \nonumber 
\end{align}
Assume that these reactions also have identical coefficients
above the reaction arrows, 
$|\alpha_1| = |\alpha_2| = \ldots = |\alpha_M| = |\alpha|$.
Then, the corresponding \emph{fused reaction} is given by 
\begin{align}
\sum_{k = 1}^N \nu_{k} X_k & \xrightarrow[]{|\alpha|} 
\left(\nu_{k_1} + \textrm{\emph{sign}}(\alpha_1) \right) X_{k_1}
+
\ldots
+
\left(\nu_{k_M} + \textrm{\emph{sign}}(\alpha_M) \right) X_{k_M}
+ \sum_{\substack{k = 1,\\ k \ne k_1,k_2,\ldots,k_M}}^N \nu_{k} X_k.
\nonumber
\end{align}
Any network obtained by fusing reactions in the canonical \emph{CRN}
is called a non-canonical \emph{CRN}.
\end{definition}
\noindent \textbf{Remark}. 
Fusion is possible when a monomial appears
in multiple equations of~(\ref{eq:DS}) 
and is multiplied, up to sign, by identical coefficients.

\begin{definition} $($\textbf{Degree of reaction}$)$ 
\label{def:degree}
Every chemical reaction $\sum_{k = 1}^N \nu_{k} X_k 
\xrightarrow[]{\alpha} \sum_{k = 1}^N \nu_{k}' X_k $
is said to be of \emph{degree $\sum_{k = 1}^N \nu_{k}$}. 
\end{definition}

\section{Appendix: Lyapunov characteristic exponents}
\label{app:LCE}
Let $\mathbf{x}(t;\mathbf{x}_0) \in \mathbb{R}^N$
be the solution of~(\ref{eq:DS})
with initial value $\mathbf{x}(0;\mathbf{x}_0) = \mathbf{x}_0$. 
Then, the \emph{linearization} of~(\ref{eq:DS})
around $\mathbf{x}(t;\mathbf{x}_0)$ is 
the $N$-dimensional linear non-autonomous DS
\begin{align}
\frac{\mathrm{d} \mathbf{w}}{\mathrm{d} t} 
& = \nabla \mathbf{f}(\mathbf{x}(t;\mathbf{x}_0)) \mathbf{w},
\; \; \; \mathbf{w}(0) = \mathbf{w}_0,
\label{eq:DS_linearization} 
\end{align}
where $\nabla \mathbf{f}(\mathbf{x}) 
\in \mathbb{R}^{N \times N}$ is the Jacobian 
of $\mathbf{f}(\mathbf{x})$.
To measure the rate of change of $\mathbf{w}
= \mathbf{w}(t;\mathbf{x}(t;\mathbf{x}_0),\mathbf{w}_0)$
in comparison with the exponential function $\exp(t)$,
we define the following quantity~\cite{Lyapunov,Adrianova}.

\begin{definition} $($\textbf{Lyapunov characteristic exponent}$)$ 
\label{def:LCE}
The finite-time \emph{Lyapunov characteristic exponent} (\emph{LCE})
of solution $\mathbf{x}(t;\mathbf{x}_0)$ of~$(\ref{eq:DS})$ 
with respect to the perturbation $\mathbf{w}_0 \ne \mathbf{0}$
is the scalar $\lambda(\cdot; \mathbf{x}_0,\mathbf{w}_0) : \mathbb{R} \to \mathbb{R}$ given by
\begin{align}
\lambda(t; \mathbf{x}_0,\mathbf{w}_0) & \equiv \frac{1}{t} 
\ln \|\mathbf{w}(t;\mathbf{x}(t;\mathbf{x}_0),\mathbf{w}_0) \|,
\label{eq:LCE}
\end{align}
where $\mathbf{w} = \mathbf{w}(t;\mathbf{x}(t;\mathbf{x}_0),\mathbf{w}_0)$ 
is the solution of~$(\ref{eq:DS_linearization})$. Assuming existence, 
the corresponding infinite-time \emph{LCE} is given by
$\lambda_{\infty}(\mathbf{x}_0,\mathbf{w}_0)
\equiv \lim_{t \to \infty} 
\lambda(t; \mathbf{x}_0,\mathbf{w}_0)$.
\end{definition} 

If the solution $\mathbf{x}(t;\mathbf{x}_0)$ is bounded
for all $t \in \mathbb{R}$, then it has at most $N$ distinct
infinite-time LCEs, all of which are finite~\cite{Adrianova};
furthermore, if $\lim_{t \to \infty} \mathbf{x}(t;\mathbf{x}_0)$ is not an equilibrium, then one of the LCEs is zero~\cite{LCE_Zero}.
In addition, infinite-time LCEs 
are invariant under affine change of coordinates~\cite{Adrianova}.
Furthermore, under suitable conditions~\cite{Adrianova,LCE_Survey}, 
trajectories in chaotic attractors
have a negative, zero and a positive infinite-time LCE, 
and the sum of the LCEs is negative.

\textbf{Computation of LCEs}.
Let matrix $X(t;\mathbf{x}_0) \in \mathbb{R}^{N \times N}$
contain as its $i$th column 
solution $\mathbf{w}(t;\mathbf{x}(t;\mathbf{x}_0),\mathbf{e}_i) 
\in \mathbb{R}^N$ of~(\ref{eq:DS_linearization}) with initial 
condition $\mathbf{w}_0 = \mathbf{e}_i$, the $i$th standard 
basis vector in $\mathbb{R}^N$; then $X(t;\mathbf{x}_0)$
is called the \emph{matriciant} (fundamental matrix) of~(\ref{eq:DS_linearization})~\cite{Adrianova}.
Having linearly independent columns, 
it admits the $Q R$-factorization
$X(t) = Q(t) R(t)$, 
with $Q(t) \in \mathbb{R}^{N \times N}$ orthogonal
and $R(t) \in \mathbb{R}^{N \times N}$ upper-triangular
with non-negative diagonal elements 
$R_{i,i}(t) \ge 0$. Under suitable regularity conditions,
the $N$ infinite-time LCEs are given by
$\lambda_{\infty,i} = \lim_{t \to \infty} 
(1/t) \ln R_{i,i}(t)$~\cite{Adrianova,LCE_Survey},
and can therefore be approximated with
$\lambda_{i}(t) = (1/t) \ln R_{i,i}(t)$ 
with $t$ sufficiently large.

However, the elements of $X(t;\mathbf{x}_0)$
can change exponentially fast with time, and can
therefore take values smaller (underflow) 
or larger (overflow) in magnitude 
than representable on a computer 
for the desired values of $t$.
To this end, let us use the matriciant identity
\begin{align}
X(m \tau;\mathbf{x}_0) & = 
X(\tau;\mathbf{x}((m-1) \tau))
X(\tau;\mathbf{x}((m-2) \tau)) \ldots 
X(\tau;\mathbf{x}(\tau))
X(\tau;\mathbf{x}_0),
\label{eq:matriciant_product}
\end{align}
for every integer $m > 0$ and time-step $\tau > 0$. 
To address the numerical issues, 
one can choose $\tau > 0$ sufficiently small, 
solve~(\ref{eq:DS}) and~(\ref{eq:DS_linearization})
to determine each of the factors on the right-hand side
of~(\ref{eq:matriciant_product}), and then 
compute the $Q R$-factorization $(m - 1)$ times iteratively as follows:  
$X(\tau;\mathbf{x}_0) = Q^1 R^1$, 
$X(\tau;\mathbf{x}(\tau)) Q^1 = Q^2 R^2$, 
$\ldots$, 
$X(\tau;\mathbf{x}((m-1) \tau)) Q^{m - 1} = Q^m R^m$,
leading to 
\begin{align}
X(m \tau;\mathbf{x}_0)
& = Q^{m} R^{m} R^{m - 1} \ldots R^2 R^1.
\end{align}
Using this composite $Q R$-factorization,
the $N$ infinite-time LCEs can be approximated with
\begin{align}
\lambda_i(t) & = \frac{1}{t} \ln \left(\prod_{j = 1}^{m} R_{i,i}^{j} \right)
= \frac{1}{t} \sum_{j = 1}^{m} \ln R_{i,i}^{j},
 \; \; \; i = 1,2,\ldots,N,
\label{eq:LCE_numerical}
\end{align} 
with sufficiently large $t = m \tau$ and 
sufficiently small $\tau > 0$. In this paper,
we use this discrete $Q R$ algorithm to compute the LCEs;
see e.g.~\cite{LCE_Survey,LCE_Error} for more details.

\end{document}